\documentclass[a4paper,11pt]{article}
\usepackage[utf8x]{inputenc}
\usepackage{amsmath}
\usepackage{amssymb}
\usepackage{amsfonts}
\usepackage{amsthm}
\usepackage[alphabetic]{amsrefs}
\usepackage{a4wide}
\usepackage[pagebackref=true]{hyperref}
\usepackage{authblk}
\usepackage[usenames, dvipsnames]{color}

\renewcommand{\labelenumi}{(\roman{enumi})} 
\newtheorem{Zae}{Zae}[section] 
\newtheorem{lemma}[Zae]{Lemma} 
\newtheorem{prop}[Zae]{Proposition}
\newtheorem{theorem}[Zae]{Theorem} 
 
\newtheorem{coro}[Zae]{Corollary} 
\theoremstyle{definition}
\newtheorem{definition}[Zae]{Definition} 
\newtheorem{remark}[Zae]{Remarks}
\newtheorem{example}[Zae]{Example}
\newcommand{\ttau}{\overline{\tau}}
\newcommand{\UU}{\overline{U}}
\newcommand{\la}{\langle}
\newcommand{\ra}{\rangle}
\newcommand{\Aut}{\mathrm{Aut}}
\newcommand{\Inn}{\mathrm{Inn}}
\newcommand{\con}{\mathrm{con}}
\newcommand{\Sym}{\mathrm{Sym}}
\DeclareMathOperator{\chara}{char}
\newcommand{\PGL}{\mathrm{PGL}}
\newcommand{\PGaL}{\mathrm{P}\Gamma\mathrm{L}}
\newcommand{\PSL}{\mathrm{PSL}}
\newcommand{\SL}{\mathrm{SL}}
\newcommand{\SU}{\mathrm{SU}}
\newcommand{\ZZ}{\mathbf{Z}}
\newcommand{\NN}{\mathbf{N}}
\newcommand{\FF}{\mathbf{F}}
\newcommand{\QQ}{\mathbf{Q}}
\newcommand{\RR}{\mathbf{R}}
\newcommand{\CC}{\mathbf{C}}
\newcommand{\EE}{\mathbf{E}}
\newcommand{\KK}{\mathbf{K}}
\newcommand{\HH}{\mathfrak{H}}
\newcommand{\e}{{\mathfrak{e}}}
\newcommand{\f}{{\mathfrak{f}}}

\newcommand{\oa}{\overline{a}}
\newcommand{\ob}{\overline{b}}

\newcommand{\ove}{\overline{e}}
\newcommand{\ox}{\overline{x}}
\newcommand{\oz}{\overline{z}}
\newcommand{\oy}{\overline{y}}
\newcommand{\iob}{\iota(\overline{b})}

\title{Boundary Moufang trees\\ with abelian root groups of characteristic $p$}
\author[1]{Pierre-Emmanuel Caprace\thanks{F.R.S.-FNRS research associate, supported in part by the ERC (grant \#278469)}}
\author[1]{Matthias Gr\"uninger\thanks{Supported by the ERC (grant \#278469)}}

\affil[1]{UCLouvain, 1348 Louvain-la-Neuve, Belgium}

\date{June 20, 2014}

\begin{document}
\maketitle

\begin{abstract}
We prove that Moufang sets with abelian root groups arising at infinity of a locally finite tree all come from rank one simple algebraic groups over local fields.
\end{abstract}

\tableofcontents
\section{Introduction}

A \textbf{Moufang set} is a pair $(X, (U_x)_{x \in X})$ consisting of a set $X$ and a collection of permutation groups $U_x \leq \Sym(X)$, called \textbf{root groups},  such that for all $x \in X$, the group $U_x$ fixes $x$, acts regularly on $X \setminus \{x\}$, and preserves the set $\{U_y \; | \; y \in X\}$ under conjugation. A key example is the set $G/P$, where $G$ is the group of $k$-points of a simple algebraic group defined over a field $k$, and of $k$-rank one, and $P$ is the group of $k$-points of a minimal $k$-parabolic subgroup. 
Moufang sets have been introduced by J.~Tits \cite{Tits} and F.~Timmesfeld \cite{Tim}, in an attempt at axiomatizing the subgroup combinatorics of simple algebraic groups of rank one. An obvious class of Moufang sets is provided by sharply $2$-transitive permutation groups. A Moufang set which is not of that form  is called \textbf{proper}. It is known that every quadratic Jordan division algebra gives rise to a Moufang set with abelian root groups, see \cite[Theorem~4.2]{DW}. A major conjecture on Moufang sets predicts that, conversely, every proper Moufang set with abelian root groups arises in this way. For a survey on partial results in the direction, we refer to \S7.6 in \cite{DS1}; since then, a major breakthrough was accomplished by Y.~Segev in \cite{S2}. The goal of the present paper is to prove the conjecture in case the Moufang set is the set of ends of a locally finite tree $T$, and the root groups are closed subgroups of $\Aut(T)$. That case is relevant both to the theory of Moufang sets, and to the structure theory of locally compact groups.  

Let $T$ be a locally finite tree which is \textbf{thick}, i.e. all of whose vertices have degree~$\geq 3$. Let $(U_{\mathfrak{e}} \; | \; {\mathfrak{e}} \in \partial T)$ be a \textbf{system of root groups} for $T$, i.e. a collection of subgroups $U_{\mathfrak{e}} \leq \Aut(T)$ such that  for all ${\mathfrak{e}} \in \partial T$, the group $U_{\mathfrak{e}}$ fixes ${\mathfrak{e}}$, acts regularly on $\partial T \setminus \{{\mathfrak{e}}\}$,  and preserves $\{U_{\mathfrak{f}} \; | \; {\mathfrak{f}} \in \partial T\}$ under conjugation. 

\begin{theorem}\label{thm:main}
Assume that the root groups  are  abelian, and are closed subgroups of $\Aut(T)$ endowed with the compact-open topology. 

Then there exists a non-Archimedean local field $k$, a simply connected $k$-simple algebraic group  $\mathbf G$ defined over $k$, and a continuous central homomorphism $\varphi \colon \mathbf G(k) \to \Aut(T)$ such that $T$ is equivariantly isomorphic to the Bruhat--Tits tree of $\mathbf  G(k)$, and $\varphi(\mathbf G(k)) = G^\dagger$, where $G^\dagger =  \la U_{\mathfrak{e}} \; | \; {\mathfrak{e}} \in \partial T\ra$ is the little projective group of the given Moufang set. In particular $G^\dagger$ is a closed subgroup of $\Aut(T)$.  
\end{theorem}

We emphasize that the isotropic simple algebraic groups over local fields have been classified by Kneser (in characteristic~$0$) and Bruhat--Tits (in positive characteristic), see \cite{Tits_corvallis}. Following \cite{Tits_boulder} and \cite{Tits_corvallis}, it turns out that the groups $\mathbf G(k)$ satisfying the conditions of Theorem~\ref{thm:main}, namely $\mathbf G$ is simply connected, $k$-simple, of  $k$-rank one and with abelian root groups over $k$, are precisely the members of   following two families:
\begin{itemize}
\item $\mathbf G(k) \cong \SL_2(D)$, where $D$ is a finite-dimensional central division algebra over $k$. In  that case, in the notation of \cite{Tits_corvallis},  the Tits-index of $\mathbf G(k)$ is $^1 A_{2d-1, 1}^{(d)}$, where $d \geq 1$ is the degree of $D$. Moreover, the tree $T$ is regular of degree $q+1$, where $q$ equals the order of the residue field of $k$.

\item $\mathbf G(k) \cong \SU_2(D, h)$, where $D$ is the quaternion central division algebra over $k$, and $h$ is an antihermitian sesquilinear form of   Witt index~$1$ relative to an involution $*$ of the first kind of $D$ such that the space   $\{x \in D \; | \; x^* = x\}$ has dimension~$3$. In  that case the Tits-index of $\mathbf G(k)$ is $C_{2, 1}^{(2)}$. Moreover, the tree $T$ is semi-regular of degree $(q+1, q^2+1)$, where $q$ equals the order of the residue field of $k$.
\end{itemize}

The proof of Theorem~\ref{thm:main} can be outlined as follows. 
It is known that if a Moufang set has abelian root groups, then the root group is either torsion-free and uniquely divisible, or  of exponent $p$ for some prime $p$, see \cite[Proposition~7.2.2]{DS1}. In other words, the root group carries the structure of a vector space over a prime field. In characteristic~$0$, i.e. when the root group is torsion-free, Theorem~\ref{thm:main} follows from \cite[Theorem~B]{CDM}. The latter relies on  deep results on $p$-adic analytic groups   due to Lazard--Lubotzky--Mann, that were used to prove that  the closure of the subgroup of $\Aut(T)$ generated by the root groups must actually be an algebraic group over $\QQ_p$. The present paper focuses on the case of positive characteristic. The strategy of proof is necessarily different, since no analogue of the aforementioned characterizations of $p$-adic analytic groups is known to hold in positive characteristic. The main idea behind our proof is the following. Given two distinct points ${\mathfrak{e}}, {\mathfrak{f}} \in \partial T$, we construct a countable collection of finite root subgroups $U^+_m \leq U_{\mathfrak{e}}$ and $U^-_n \leq U_{\mathfrak{f}}$, where $m, n \in \ZZ$, generating a dense subgroup $\Gamma$ of the locally compact group $G = \overline{ \la U_{\mathfrak{e}} \cup U_{\mathfrak{f}} \ra} \leq \Aut(T)$, and forming a so-called \textbf{RGD-system}, as defined by J.~Tits \cite{Tits}.  This means in particular that the tree $T$ admits a twin tree $T'$ on which $\Gamma$ acts by automorphisms, and such that the diagonal $\Gamma$-action on $T \times T'$ preserves the twinning. The presence of the finite root groups $U^+_m$ and $U^-_n$ ensures that the twinning satisfies the Moufang condition. At that point, we have obtained a Moufang twin tree, whose set of ends forms a Moufang set with abelian root groups. Those Moufang twin trees have been studied comprehensively by the second author in \cite{GrMoufangTwinTrees} (without any hypothesis of local finiteness). The conclusion of Theorem~\ref{thm:main} can then be deduced from the main result of loc. cit.

It should be pointed out that the group $\Gamma$ constructed along the proof happens to be a non-uniform lattice in the product group $G \times \Aut(T')$. The fact that our analysis of the class of locally compact groups $G$ under consideration has involved the construction of a lattice $\Gamma$ in $G  \times \Aut(T')$ is an interesting feature: the familiar approach consisting in reducing the study of lattices to the study of their ambient locally compact groups has been reversed here.

\medskip

We now describe applications of Theorem~\ref{thm:main}. The first provides a criterion allowing one to identify rank one algebraic groups with abelian root subgroups among locally compact groups with a boundary transitive action on a locally finite tree. 

Let $G \leq \Aut(T)$ be a closed non-compact subgroup acting transitively on $\partial T$. Besides rank one algebraic groups over non-Archimedean local fields, there are several other sources of groups $G \leq \Aut(T)$ as above: complete Kac--Moody groups of rank two over finite fields, the full automorphism group $G= \Aut(T)$ in case $T$ is regular or semi-regular, and many variations such as groups with prescribed local actions as defined any studied by Burger--Mozes \cite{BurgerMozes}. 

General results by Burger--Mozes, valid in all cases, imply that the group $G^{(\infty)}$, defined as the intersection of all non-trivial closed normal subgroups of $G$, is a compactly generated, topologically simple closed subgroup of $\Aut(T)$ acting doubly transitively on $\partial T$, and such that $G/G^{(\infty)}$ is compact (see \cite[Theorem~2.2]{CDM}). 

\begin{coro}\label{cor:con}
Let $T$ be a thick locally finite tree and $G \leq \Aut(T)$ be a closed non-compact group acting transitively on $\partial T$. 

 Then the following assertions are equivalent. 
\begin{enumerate}
\item  There exists a hyperbolic automorphism  $h \in G$ such that the \textbf{contraction group} $\con(h) = \{g \in G \; | \; \lim_n h^n g h^{-n}=1\}$ is abelian. 

\item There exists a non-Archimedean local field $ k$, a simply connected $k$-simple algebraic group  $\mathbf G$ defined over $k$ with abelian root groups over $k$, and a continuous central surjective homomorphism $\varphi \colon \mathbf G(k) \to G^{(\infty)}$. In particular $G$ is isomorphic to a closed subgroup of  $ \Aut(\mathbf G(k))$ containing $\Inn(\mathbf G(k)) \cong \mathbf G(k)/Z \cong G^{(\infty)}$. 
\end{enumerate}
\end{coro}

We may combine Corollary~\ref{cor:con} with general results on the structure of locally compact groups from \cite{CCMT} to derive the following (compare Corollary~D from \cite{CDM}). 

\begin{coro}\label{cor:unimod}
Let $G \neq \{1\}$ be a unimodular locally compact group whose only compact normal subgroup is the trivial one. Assume that for some $h \in G$, the contraction group $\con(h)$ is abelian,  and that the closed subgroup $\overline{\la h \con(h) \ra}$ is cocompact in $G$. 

Then there exists a (possibly Archimedean)  local field $k$, a simply connected $k$-simple algebraic group  $\mathbf G$ defined over $k$ and a continuous central homomorphism $\varphi \colon \mathbf G(k) \to G$ such that $\varphi(\mathbf G(k))$ is the smallest closed normal subgroup of $G$. In particular $G$ is isomorphic to a closed subgroup of  $\Aut(\mathbf G(k))$ containing $\Inn(\mathbf G(k))$.
\end{coro}

Another application concerns sharply $3$-transitive permutation groups. 
J.~Tits \cite{Tits3trans} proved that if a Lie group $G$ has a continuous sharply $3$-transitive action on a topological space $X$, then $G= \PGL_2(\RR)$ or $ \PGL_2(\CC)$ and $X$ is the projective line $\mathbb P^1(\RR)$ or $\mathbb P^1(\CC)$. One expects more generally that the only source of sharply $3$-transitive actions of a $\sigma$-compact locally compact group $G$ on a  compact space $X$ is given by the action of a group $G$ with $\PSL_2(k) \leq G \leq \PGaL_2(k)$ over $\mathbb{P}^1(k)$, where $k$ is a local field. Partial progress towards this conjecture has recently been accomplished by Carette--Dreesen \cite{CarDre}, showing that if $G$ is not a Lie group, then $G$ has a continuous proper action by automorphisms on a regular locally finite tree $T$, such that $X$ is equivariantly homeomorphic to the set of ends $\partial T$. We are then very close to a situation where Theorem~~\ref{thm:main} applies; we indeed obtain the following. 

\begin{coro}\label{cor:sharply3transitive}
Let $G$ be a $\sigma$-compact locally compact group and $\Omega$ be a compact $G$-space, such that the $G$-action on $\Omega$ is sharply $3$-transitive. Then the following assertions are equivalent. 
\begin{enumerate}
\item For $\omega \in \Omega$, there exists a normal subgroup $U_\omega \leq G_\omega$ 
 acting regularly on $\Omega \setminus \{\omega\}$. 

\item There exists a (possibly Archimedean)  local field $k$  and a group $
\PSL_2(k) \leq H \leq \PGaL_2(k)$ such that $(G, \Omega)$ and $(H, \mathbb{P}^1(k))$ are 
isomorphic topological transformation groups. 

\end{enumerate}
\end{coro}

In certain cases, the condition (i) in Corollary~\ref{cor:sharply3transitive} is known to be automatically satisfied. Indeed, given any (abstract) sharply $3$-transitive group $G$, there is a way to define a \textbf{permutational characteristic} of $G$, which is either $0$ or a prime number $p$. It is known  that if the characteristic of $G$ is either $3$ or congruent to $1$ modulo $3$, then the condition (i) in Corollary~\ref{cor:sharply3transitive} automatically holds (see \cite[\S13]{Kerby} and Corollary~\ref{cor:1modulo3} below).

\section{Preliminaries}

\subsection{Generalities on Moufang sets}

We briefly recall some general conventions and notation from the basic theory of Moufang sets; an excellent exposition of the material can be found in \cite{DS1}. 

Given a group $H$, the set of non-trivial elements of $H$ is denoted by $H^\#$. 

Let $(X, (U_x)_{x \in X})$ and pick any point $\infty \in X$.  The group $G^\dagger = \la U_x \; | \; x \in X \ra$ is called the \textbf{little projective group} of the Moufang set. It is doubly transitive on $X$. 
Given two elements $0, \infty \in X$ and any element $\tau \in G^\dagger$ swapping $0$ and $\infty$, we may identify  the set $X \setminus \{\infty\}$ with the root group $U_\infty$ by sending $x \in X \setminus \{\infty\}$ to the unique element $\alpha_x \in U_\infty$ such that $0 \alpha_x = x$. Under this identification, the permutation $\tau$ may be viewed as a permutation of the set $(U_\infty)^\#$. It turns out that the whole Moufang set $(X, (U_x)_{x \in X})$ can be recovered from the pair $(U_\infty, \tau)$. This Moufang set is denoted by $\mathbb  M(U_\infty, \tau)$. 

For example, if $X$ is the projective line over a field $\FF$ and $U_x \leq G = \PGL_2(\FF)$ is the unipotent radical of the stabiliser $G_x$, then $(X, (U_x)_{x \in X})$ is a Moufang set. It can be identified with the Moufang set $\mathbb  M(U, \tau)$, where $U$ is the additive group of $\FF$ and $\tau \colon \FF^\# \to \FF^\# : x \mapsto -x^{-1}$. This Moufang set will be denoted by $\mathbb  M(\FF)$. 

We now come back to a general Moufang set $\mathbb  \mathbb  M(U_\infty, \tau)$. Set $U = U_\infty$. It is customary to denote the group $U$ additively, even though it is not necessarily commutative. For any $a \in U^\#$, there is a unique element $\mu_a$ of the double coset $(U^\tau) a (U^\tau)$ which swaps $0$ and $\infty$. The element $h_a = \tau \mu_a \in G^\dagger$ is called the \textbf{Hua map} associated with $a$. The subgroup  $\HH$ of $G^\dagger$ generated by the Hua maps is called the \textbf{Hua group}. The Moufang set is \textbf{proper} if  $\HH$ is non-trivial; this means equivalently that $G^\dagger$ is not sharply $2$-transitive on $X$. 

The Moufang set $\mathbb  \mathbb  M(U, \tau)$ is called \textbf{special} if $-a \tau = (-a)\tau$ for all $a \in U^\#$. A fundamental result due to Y. Segev \cite{S2} ensures that if $U$ is abelian and $\mathbb  \mathbb  M(U, \tau)$ is proper, then it is special. 

The following basic identities will be needed. 
\begin{lemma}\label{mumaps}
Let $\mathbb M(U,\tau)$ be a Moufang set and $a \in U^{\#}$. Then we have
\begin{enumerate}
\item $\mu_a^{-1} =\mu_{-a}$.
\item $\mu_{ah} = \mu_a^h$ for all $h \in \HH$.
\item If $\mathbb M(U,\tau) = \mathbb M(U,\tau^{-1})$, then $\mu_{a\tau} = \mu_{-a}^{\tau}$.
In particular $\mu_{a\mu_b} =\mu_{-a}^{\mu_b}$ for all $b \in U^\#$.
\item If $\mathbb M(U,\tau)$ is special, then $a\mu_a = -a$ and $(-a)\mu_a =a$.
\end{enumerate}
\end{lemma} 
\begin{proof}
See 4.3.1 and 7.1.4 of \cite{DS1}.
\end{proof}
\begin{lemma}\label{basic formulas}
Let $\mathbb  M(U,\tau)$ be a Moufang set and $a,b \in U^{\#}$ with $a \ne b$. Then the following assertions hold.
\begin{enumerate}
\item $(a\tau^{-1} -b\tau^{-1})\tau = (a-b)\mu_b + (-b\tau^{-1})\tau$.
\item $\mu_{-b}\mu_{b-a} \mu_a = \mu_{(a\tau^{-1} -b\tau^{-1})\tau}$.
\end{enumerate}
\end{lemma}
\begin{proof}
See 6.1.1 in \cite{DS1}.
\end{proof}

When the root group $U$ is abelian, it is customary to define $h_a$ to be the constant zero map on $U$ if $a = 0$. In this way, for all $a, b \in U$, we obtain a well-defined endomorphism
$$h_{a, b} \colon U \to U : x \mapsto xh_{a+b} - x h_a - xh_b.$$
Notice that $h_{a, b} = h_{b, a}$. 

\begin{lemma}\label{hab}
Let $\mathbb  M(U,\tau)$ be a special Moufang set with $U$ abelian. Then $a\tau h_{a,b}=-2b$ for all
$a \in U^{\#}$ and all $b \in U$.
\end{lemma}
\begin{proof}
See 5.1 in \cite{DW}.
\end{proof}

\begin{lemma}\label{mua=mub} Let $\mathbb  M(U,\tau)$ be a special Moufang set. 
\begin{enumerate}
\item If $a,b \in U^{\#}$ with $\mu_a =\mu_b$, then $b \in \{a,-a\}$.
\item $\mu_a = \mu_{-a}$ for all $a \in U^{\#}$ if and only if $U$ is abelian. 
\end{enumerate}
\end{lemma}
\begin{proof}
Part (i) is 4.9(4) from \cite{DS3}, while Part (ii) is 6.3 from \cite{DST}.
\end{proof}

\begin{lemma}\label{biadditivity} Let $\mathbb  M(U,\tau)$ be a special Moufang set with $U$ abelian and 
  $\tau =\mu_e$ for some $e \in U^{\#}$. Then for all $a,b,c \in U$ we have the following.
  \begin{enumerate}
  \item $h_{a+b,c} -h_{a,c} -h_{b,c} = h_{a,b+c} -h_{a,b} -h_{a,c}$.

  \item If $a \ne 0$, then $a\tau h_{a+b,c} = -2c  +a\tau h_{b,c}$.
  \end{enumerate}
\end{lemma}
  
\begin{proof}
(i) The following argument is similiar to that used in the proof of 5.10 from \cite{DS3}.
  We have 
\begin{align*}
h_a + h_b + h_{a,b} + h_c + h_{a+b,c} &=  h_{a+b} + h_c + h_{a+b,c} \\ 
& =  h_{a+b+c}  \\
 & =  h_a+h_{b+c} +h_{a,b+c}  \\
 & =   h_a +h_b + h_c +h_{b,c}+h_{a,b+c}
\end{align*}
and hence $h_{a+b,c} + h_{a,b} = h_{b,c} +h_{a,b+c}$ as requested.
 
\medskip \noindent
(ii) Using Part (a) and Lemma~\ref{hab}, we obtain
\begin{align*}
0  &=   -2(b+c) +2b +2c \\
& =   a\tau h_{a,b+c} -a\tau h_{a,b} -a\tau h_{a,c} \\
&=  a\tau h_{a+b,c} -a\tau h_{a,c} -a\tau h_{b,c}  \\
& =  a\tau h_{a+b,c} +2c -a\tau h_{b,c}
\end{align*}
and hence $a\tau h_{a+b,c} = -2c + a\tau h_{b,c}$, as requested. 
\end{proof}

 Let $\mathbb   M(U,\tau)$ be a Moufang set. A \textbf{root subgroup} is a subgroup $V \leq U$ such that for some $a \in V^\#$, we have $v \mu_a \in V$ for all $v \in V^\#$.  In that case $\mathbb   M(V,\mu_a |_{V^\#})$  is itself a Moufang set, which is regarded as a Moufang subset of $\mathbb   M(U,\tau)$.

  \begin{lemma}\label{root subgroup}
 Let $\mathbb  M(U,\tau)$ be a special Moufang set with $U$ abelian. 
 Let $V$ be a root subgroup of $U$ such that the
 corresponding Moufang subset is isomorphic to $\mathbb {M}(\FF)$, with $\FF$ a finite field.

Then the group $H_0=\langle \mu_a \mu_b\; | \;  a,b \in V^{\#} \rangle$  acts faithfully on $V$. In particular, it is 
isomorphic to the cyclic group $(\FF^*)^2$. 
\end{lemma}
\begin{proof}
Let $m$ be the order of $(\FF^*)^2$. Then there are elements $a,b \in V^{\#}$ such that the image of $h=\mu_a 
 \mu_b$ in $H_0/C_{H_0}(V)$ has  order $m$. For an even integer $i \in \NN$, we deduce from Lemma~\ref{mumaps} that
 $$h^i =(\mu_a \mu_b)^i= \mu_a h^{-\frac{i}{2}} \mu_a h^{ \frac{i}{2}}
 =\mu_a \mu_{ah^{\frac{i}{2}}},$$
while for an odd integer $i\in \NN$, we have 
  $$h^i =(\mu_a \mu_b)^i =\mu_a h^{-\frac{i-1}{2} } \mu_b h^{\frac{i-1}{2}} = 
 \mu_a \mu_{bh^{\frac{i-1}{2}}}.$$  
Now we set $c =ah^{\frac{m}{2}}$ for $m$ even and $c=bh^{\frac{m-1}{2}}$ for $m$ odd. 
 Then $h^m = \mu_a \mu_c \in C_{H_0}(V)$. This implies that $c \in \{a,-a\}$ by Lemma~\ref{mua=mub}.
   Thus $h^m =1$. 
Since $|\FF^*:(\FF^*)^2|\leq 2$, we have $V^{\#} = aH_0 \cup bH_0$, so that $\mu_c \mu_d 
  =(\mu_a \mu_c)^{-1} \mu_a \mu_d \in \langle h \rangle$ for all $c,d \in V^{\#}$. It follows 
  that $H_0 =\langle h \rangle$ is cyclic of order $m$.
 \end{proof}
 
\subsection{The special role of the field of order~$9$}

The following lemma also appears in \cite{BGM}.

\begin{lemma}\label{Galois automorphism}
Let $\FF$ be a field and $\sigma \in 
\Aut(\FF)$ be an automorphism such that $x^{-1} x^{\sigma} \in \mathbf{K}=Fix_{\sigma}(\FF)$ for all $x \in 
(\FF^*)^2$.
Then $\sigma =1$ or $\FF =\FF_9$. 
\end{lemma}
\begin{proof}
Set $\tau =\sigma^2$ if $\chara(\FF)\neq 2$ and $\tau =\sigma$ if $\chara(\FF)=2$. 
For any $x \in \FF^*$ we have $(x^{-1} x^{\sigma})^{\sigma} \in \{x^{-1} x^{\sigma}, -x^{-1} x^{\sigma}\}$ 
because  $\sigma$ fixes $(x^{-1} x^{\sigma})^2$. Therefore
$\tau$ fixes $x^{-1} x^{\sigma}$.
Since $x^{-1} x^{\sigma^2} = x^{-1} x^{\sigma} (x^{-1} x^{\sigma})^{\sigma}$ we conclude that
in every case $x^{-1} x^{\tau} \in \mathbf{L}:=Fix_{\tau}(\FF)$ for all $x \in \FF^*$. 
For $x \in \FF \setminus \{0,-1\}$ we have
$$x^{-1} x^{\tau} -(x+1)^{-1} (x+1)^{\tau} = x^{-1}(x+1)^{-1}(x^{\tau} -x) \in \mathbf{L}.
$$
Moreover, since $x^{-1} x^{\tau^2} = x^{-1} x^{\tau} (x^{-1} x^{\tau})^{\tau} \in \mathbf{L}$, we have 
$$x^{-\tau^{-1}}(x^{\tau} -x) = (x^{-1}  x^{\tau^2} -x^{-1} x^{\tau})^{\tau^{-1}} \in 
\mathbf{L}.$$
If $x \ne x^{\tau}$, we obtain $x^{-\tau^{-1}} x (x+1) \in \mathbf{L}$ and hence $x^{-1} x^{\tau} (x+1)^{\tau} 
\in \mathbf{L}$. But this implies $x \in \mathbf{L}$, a contradiction. Therefore $\tau =1$.

We are therefore left with the case $\chara(\FF) \ne 2$ and $\sigma^2 =1$. In that case, there is 
an element $\omega \in \FF $ with $\omega^2 \in \mathbf{K}$, 
$\omega^{\sigma} =-\omega$ and $\FF=\mathbf{K}(\omega)$. For $\lambda \in \mathbf{K}^*$, we obtain
$$c= (\lambda +\omega)^{-2} (\lambda -\omega)^2 =(\lambda+\omega)^{-2+2\sigma} \in \mathbf{K}$$
and hence 
$$c(\lambda^2 +2\lambda \omega +\omega ^2) = \lambda^2 -2\lambda \omega +\lambda^2.$$
Hence $(2+2c)\lambda \omega = (1-c)(\omega^2 +\lambda^2)$. Since the right hand side belongs to
$\mathbf{K}$, this implies $\omega=0$ (and hence 
$\mathbf{K}=\FF$ and $\sigma =1$) or $c=-1$. But then $2(\omega^2+\lambda^2)=0$ for all $\lambda \in 
\mathbf{K}$, 
so $(\mathbf{K}^*)^2 =\{1\}$. Thus if $\sigma \ne 1$, then $\mathbf{K}=\FF_3$ and $\FF=\FF_9$.
\end{proof}

We emphasize that   $\FF_9$ is indeed a genuine exception in Lemma~\ref{Galois automorphism}. Indeed, denoting the non-trivial automorphism of $\FF_9$ by $\sigma$, we have
$(x^{-1} x^{\sigma})^2 = (x^{-1} x^3)^2 = x^4 \in \FF_3$ for all $x \in \FF_9^*$.
The special situation for $\FF=\FF_9$ can be explained by the fact that $Mat_2(\FF_3)$ 
contains exactly three subfields of order $9$ whose squares normalize each other. In fact, the group $\mathrm{GL}_2(\FF_3)$ has a 
normal subgroup isomorphic to $Q_8$ and every cyclic subgroup of order $4$ of this normal subgroup generates 
a field of order $9$. That fact will be further exploited in Lemma~\ref{F9} below.

\medskip
We recall from \cite{G} that a \textbf{multiplicative quadratic map} is a map $q \colon \mathbf K \to \mathbf L$ between two unital rings satisfying the following conditions:
\begin{itemize}
\item $ q(ab) = q(a) q(b)$ for all $a, b \in K$.

\item $q(n) = n^2$ for all $n \in \ZZ$.

\item The map $f \colon K \times K \to L : (a, b) \mapsto q(a+b) - q(a) - q(b)$ is biadditive.
\end{itemize}
 
The following result, due to the second author, will play an important role here. 

\begin{theorem} \label{thm:MQM}
Let $K, L$ be commutative fields, and $q \colon \mathbf K \to \mathbf L$ be a multiplicative quadratic map. Then one of the following holds. 
\begin{enumerate}
\item There exists a pair $\{\phi_1, \phi_2\}$ of field monomorphisms from $K$ to   $L$ such that $q(a) = a^{\phi_1} a^{\phi_2}$ for all $a \in K$. 

\item There exists a separable quadratic extension $M/L$ and a field monomorphism $\varphi \colon K \to M$ such that $q(a) = N_{M/L}(a^{\varphi})$  for all $a \in K$. 

\item $K$ and $L$ both have characteristic~$2$, and $q$ is a field monomorphism. 
\end{enumerate}

\end{theorem}
\begin{proof}
Follows from Theorem~1.2 in \cite{G} and \cite[1.6.2]{SpringerVeldkamp}.
\end{proof}

\begin{coro}[Corollary~1.3 in \cite{G}]\label{cor:MQM}
Let $K, L$ be commutative fields, and $q \colon \mathbf K \to \mathbf L$ be a multiplicative quadratic map. Then there exists a unique pair $\{\phi_1, \phi_2\}$ of field monomorphisms from $K$ to the algebraic closure $\overline L$ such that $q(a) = a^{\phi_1} a^{\phi_2}$ for all $a \in K$. \qed
\end{coro}

In the following lemma, we regard $\FF_9$ as a subring of $End_{\FF_3}(\FF_9)$ via the natural embedding 
$x \mapsto (y \mapsto yx)$. 
\begin{lemma}\label{F9} 
Let $\mathbf{K}$ be a field and $q: \mathbf{K} \to End_{\FF_3}(\FF_9)$ be a multiplicative 
quadratic map such 
that $q(\mathbf{K}^*)$ and $(\FF_9^*)^2$ normalize each other. If $q(\mathbf{K})$ is not contained in 
$\FF_9$, then $\mathbf{K}$ has order $9$ and $q(\mathbf{K})$ generates a field isomorphic to $\FF_9$.
\end{lemma} 
\begin{proof}
Let $\sigma$ be the non-trivial Galois automorphism of $\FF_9$ and let $\Gamma =\FF_9^* \rtimes \langle \sigma
 \rangle$, viewed as the subgroup of the unit group of  $End_{\FF_3}(\FF_9)$. Notice that $\Gamma$ coincides with the normaliser of $\FF_9^*$ in that unit group. Therefore, by assumption $q(\mathbf{K}^*)$ is contained in $\Gamma$ but not in $\FF_9^*$. 
 
Let now $x \in \mathbf{K}^*$ be  an element with $q(x)\not\in \FF_9^*$ and $i\in \FF_9^*$ be an element of order $4$. Then 
$-1 = i^2= i^{-1} i^3 = i^{-1} i^{\sigma} =  i^{-1} i^{q(x)} = [i,q(x)] \in q(\mathbf{K}^*)$. Since 
$q(\mathbf{K}^*)$ is abelian, we have two possibilities:
\begin{enumerate}
\item There is an element $\epsilon \in \FF_9^*$ of order $8$ such that
$q(\mathbf{K}^*) = \langle \varphi \rangle$ with $x\varphi = x^3 \epsilon$ for all $x \in \FF_9$. 
\item $q(\mathbf{K}^*) = \langle -1, \sigma \rangle \cong V_4$. 
\end{enumerate} 
In the first case $(1 +\varphi) \in End_{\FF_3}(\FF_9)$ is invertible and has order $8$, so the subring $\mathbf{L}$ generated 
by  $\varphi$ is a field of order $9$ and $q: \mathbf{K} \to \mathbf{L}$ is a multiplicative quadratic map, so 
either $|\mathbf{K}|=9$ or $\mathbf{K} \cong \FF_{81}$ and $q(x) =x^{10}$ by Theorem~\ref{thm:MQM}. But in this case 
$q(\mathbf{K}^*)$ contains  an element of order $8$, a contradiction.

Suppose that the second case holds. Then $|\mathbf{K}^*:(\mathbf{K}^*)^2| \geq 4$, so $\mathbf{K}$ is an infinite 
field of characteristic $3$. Thus there are $a,b \in \mathbf{K}^*$ with $a^4 \ne b^4$.
Set $x = (a-b)^2, y =(2ab)^2 $ and $z =(a+b)^2$. 
Then $x+y =z$ and $q(x) =q(y) =q(z) =1$, so $f(x,y) =q(z) -q(x) -q(y) =-1$ and 
hence $q(a^4 -b^4) = q(x-y) =q(x) +q(y) +f(x,-y) =1+1+1 =0$, a contradiction. 
\end{proof}

\subsection{Twin trees and RGD-systems}\label{sec:RGD}

Twin trees were introduced in \cite{RT94} which we recommend as a reference.

\begin{definition} Let $T_+,T_-$ be two trees such that every vertex has at least $3$ neighbours.
 \begin{enumerate}
 \item A symmetric function $\delta^*:T_+ \times T_- \cup T_- \times T_+ \to \NN$ is called a 
 \textbf{codistance} if for $\epsilon \in \{+,-\}$, $v_{\epsilon}, \in T_{\epsilon}$ and $v_{-\epsilon} \in 
 T_{-\epsilon}$ the following hold:
 \begin{enumerate}
 \item For every vertex $w_{\epsilon} $ adjacent to $v_{\epsilon}$ one has
 $\delta^*(w_{\epsilon},v_{-\epsilon}) = \delta^*(v_{\epsilon},w_{\epsilon}) \pm 1$.
 \item If $\delta^*(v_{\epsilon},v_{-\epsilon}) >0$, then there is a unique vertex $w_{\epsilon}$
 adjacent to $v_{\epsilon}$ with $\delta^*(w_{\epsilon},v_{-\epsilon})=\delta^*(v_{\epsilon},
 v_{-\epsilon}) +1$. 
\end{enumerate}   
\item If $\delta^*$ is a codistance, then $(T_+,T_-,\delta^*)$ is called a \textbf{twin tree}.
 \end{enumerate}
 \end{definition}
 From \cite{RT94} it follows that $T_+$ and $T_-$ are isomorphic, semi-regular trees. Therefore one 
 also talks of a \textbf{twinning} of $T_+$. By \cite{RT99} every semi-regular tree admits 
 uncountably many twinnings.
 \begin{example}
 Let $k$ be a field and $K=k(t)$ the rational function field in one indeterminate over $k$. For 
 $\epsilon =+,-$ there is a unique valuation $v_{\epsilon}$ on $K$ which is trivial restricted to 
  $k$ such that $v_{\epsilon}(t^{\epsilon}) =1$. Let $T_{\epsilon}$ be the Bruhat--Tits tree for 
  $SL_2(K)$ with respect to the valuation $v_{\epsilon}$. As noted in \cite{RT94} there is 
  a codistance $\delta^*: T_+\times T_- \cup T_- \times T_+ \to \NN$ and so $(T_+,T_-,\delta^*)$
  is a twin tree.
 \end{example}
 An \textbf{automorphism} of a twin tree $T=(T_+,T_-,\delta^*)$ is a pair $g =(g_+,g_-) 
 \in \Aut T_+ \times \Aut T_-$ such that $\delta^*(v_+^{g_+},v_-^{g_-}) =
 \delta^*(v_+,v_-)$ for all $(v_+,v_-) \in T_+ \times T_-$. Let $\Aut T$ be the group of all 
 automorphisms of $T$.  One easily sees that $g\in \Aut T$ is 
 determined by $g_+$ (or $g_-$) and therefore regards $\Aut T$ 
 as a subgroup of $\Aut T_+$ (or $\Aut T_-$). In the example above, the 
 automorphism group of $T$ is $\PGL_2(k[t,t^{-1}]) \rtimes {\cal G}$, where 
 ${\cal G}$ is the group of field automorphisms of $K$ which stabilize the two valuations 
 $v_+$ and $v_-$.  
 
A \textbf{twin apartment} is a pair $(\Sigma_+,\Sigma_-)$ such that $\Sigma_{\epsilon}$ is an 
apartment in $T_{\epsilon}$ and such that for every vertex $v_{\epsilon}$ in $\Sigma_{\epsilon}$ 
there is a unique vertex $v_{-\epsilon} $ in $\Sigma_{-\epsilon}$ with 
$\delta^*(v_{\epsilon},v_{-\epsilon})=0$ (and therefore $\delta^*(v_{\epsilon},w_{-\epsilon})$ goes 
to infinity 
as $w_{-\epsilon}$ goes to infinity in one of the two directions of $\Sigma_{-\epsilon}$. 
A \textbf{twin root} is a pair $\alpha=(\alpha_+,\alpha_-)$ such that $\alpha_{\epsilon}$ is a
half-apartment in $T_{\epsilon}$ and such that for $(v_+,v_-) \in \alpha_+ \times \alpha_-$ one 
has $\delta^*(v_+,v_-) =0$ if and only if $v_{\epsilon}$ is the extremal vertex of 
$\alpha_{\epsilon}$ for $\epsilon \in \{+,-\}$. We call 
$(\alpha_+ \setminus \{v_+\}, \alpha_- \setminus \{v_-\})$ the \textbf{interior} of $\alpha$, 
where $v_{\epsilon}$ is the extremal vertex of $\alpha_{\epsilon}$. 
Note that if $w$ is adjacent to $v_{\epsilon}$ for $\epsilon \in \{+,-\}$, then there is a unique 
twin apartment containing $\alpha$ and $w$ (\cite{RT94}). 
Two twin roots $\alpha$ and $\beta$ are called \textbf{opposite} if $(\alpha_+ \cup \beta_+,
\alpha_- \cup \beta_-)$ is a twin apartment. 

For a twin root $\alpha$ one defines $U_{\alpha}$ as the group of those $g\in \Aut T$ which stabilize every 
vertex adjacent to a vertex contained in the interior of $\alpha$. The group $U_{\alpha}$ is called the 
\textbf{root group} corresponding 
to $\alpha$. By \cite{RT94} the group $U_{\alpha}$ acts freely on the set of twin apartments 
containing $\alpha$. We say that $U_{\alpha}$ is a \textbf{full root group} if this action is 
transitive. The twin tree $T$ has got the \textbf{Moufang property} if for every twin root 
$\alpha$ the root group $U_{\alpha}$ is full (or equivalently, if there is a twin apartment 
$\Sigma$ such that for every twin root $\alpha$ contained in $\Sigma$ the root group $U_{\alpha}$ 
is full). From \cite{RT94} it follows that if $T$ is a Moufang twin tree and $\alpha$ and 
$\beta$ are opposite twin roots of $T$, 
then $\la U_{\alpha}, U_{\beta} \ra$ is a rank one group with unipotent subgroups 
$U_{\alpha}$ and $U_{\beta}$. \\   
We repeat the definition of a RGD-system. For simplicity, we only consider those of
type $\tilde{A}_1$, although the concept can be defined for general root systems. 

\begin{definition} Let $G$ be a group, $H \leq G$ and 
$(U_n^{\epsilon})_{n \in \ZZ, \epsilon \in \{+,-\}})$ be a family of subgroups of $G$ with 
$H \leq N_G(U_n^{\epsilon})$ for all $n \in \ZZ$ and $\epsilon\in \{0,1\}$.
 Then $(G, (U_n^{\epsilon})_{n \in \ZZ, \epsilon \in \{+,-\}}, H)$ is a \textbf{RGD-system} of type $\tilde{A}_1$ if the following axioms hold:
\begin{description}
\item[(RGD0)] $U_n^{\epsilon} \ne 1$ for all $n \in \ZZ$ and $\epsilon \in \{+,-\}$.
\item[(RGD1)] $[U_n^{\epsilon},U_m^{\epsilon}] \subseteq U_{n+1}^{\epsilon} \ldots
 U_{m-1}^{\epsilon}$ for $n < m \in \ZZ$ and $\epsilon \in \{+,-\}$.
 \item[(RGD2)] For $\epsilon\in \{+,-\}, n \in \ZZ$ and $1 \ne a \in U_n^{\epsilon}$ there is a 
 unique element $\mu_a \in U_n^{-\epsilon} a U_n^{-\epsilon}$ such that $(U_m^{\delta})^{\mu_a} 
 = U_{2n-m}^{-\delta}$ for all $m \in \ZZ$ and $\delta \in \{+,-\}$.
 \item[(RGD3)] For $\epsilon \in \{+,-\}$ and $n \in \ZZ$ we have
 $U_n^{\epsilon} \cap U_{-\epsilon}=1$, where $U_{-\epsilon} =\la U_m^{-\epsilon} \; | \; m \in \ZZ,
 \epsilon \in \{+,-\}\ra$. 
 \item[(RGD4)] $G$ is generated by $H$ and $\bigcup_{n \in \ZZ, \epsilon \in \{+,-\}} U_n^{\epsilon}$.
\end{description}
\end{definition}
 If $T$ is a Moufang twin tree, $G$ the subgroup of its automorphism group generated by all root 
 groups, $\Sigma$ a twin apartment, $H$ the pointwise stabilizer of $\Sigma$ and 
 $(U_n^{\epsilon})_{n\in \ZZ, \epsilon \in \{+,-\}}$ the family of root groups corresponding to 
 the roots contained in $\Sigma$, then   $(G,H, (U_n^{\epsilon})_{n\in\ZZ, 
 \epsilon \in \{+,-\}})$ is a RGD-system of type $\tilde{A}_1$. On the other hand, every RGD-system
  of type $\tilde{A}_1$ gives rise to a Moufang twin tree of degree 
  $(|U_0^+|+1,|U_1^+|+1)$, see for example \cite{AB}. Therefore, these 
  two concepts are equivalent.
  
  \begin{example}
  In the example above, there is a natural action of $\PSL_2(k[t,t^{-1}])$ on $T=(T_+,T_-)$. 
  There is a twin apartment $\Sigma$ such that the family  $(U_n^{\epsilon})_{n\in \ZZ, 
  \epsilon \in \{+,-\}}$ of root groups associated to the twin roots contained in $\Sigma$ is defined by  
  $$U_n^+ := \left\{ \left( \begin{array}{cc} 1 & 0 \\ at^n & 1 \end{array} \right)Z; a \in 
  k\right\}
  \hspace{1cm} \text{and}   \hspace{1cm}
U_n^-:= \left\{ \left( \begin{array}{cc} 1 & at^{-n} \\ 0 & 1 \end{array} \right)Z; a \in 
  k\right\},$$
 where  $Z= Z(\SL_2(K))$, while
  $$H:=\left\{ \left( \begin{array}{cc} a & 0 \\ 0 & a^{-1} \end{array} \right)Z; a \in k^* \right\}$$ 
  is the pointwise stabilizer of $\Sigma$. The root groups are all full, 
  so $T$ is Moufang.
  \end{example}
\begin{definition} A Moufang twin tree is called a \textbf{commutative $\SL_2$-twin tree} if the following
two conditions hold:
\begin{enumerate}
\item If $\alpha, \beta$ are two opposite twin roots, then $\la U_{\alpha}, U_{\beta} \ra 
\cong \SL_2(F)$ or $\PSL_2(F)$ for a commutative field $F$.
\item If $\alpha, \beta$ are two prenilpotent roots (i.e. $\alpha_{\epsilon} \cap \beta_{\epsilon}$
is a half-apartment for $\epsilon \in \{+,-\}$), then $[U_{\alpha},U_{\beta}] =1$.
\end{enumerate}  
\end{definition}
For the corresponding RGD-system this means that every root group is isomorphic to the additive
group of a field and the commutator appearing in (RGD1) is in fact trivial.
A classification of commutative $\SL_2$-twin trees can be found in \cite{BGM}.    

In \cite{GrMoufangTwinTrees} the second author examined the following question: which commutative $\SL_2$-twin 
trees induce a Moufang set on the boundary of $T_+$ (or on the boundary of $T_-$, which is equivalent).
More precisely, let $T$ be a Moufang twin tree and $\Sigma$ be a twin apartment of $T$. Then 
$\Sigma$ has two ends in $\partial T_+$ which we denote by $\infty$ and $\mathfrak{o}$. 
For $\epsilon \in \{\infty,\mathfrak{o}\}$ we let
$U_{\epsilon}$ denote the closure of the group generated by all root group $U_{\alpha}$ with
$\alpha$ contained in $\Sigma$ and $\partial \alpha = \epsilon$. Note that $U_{\epsilon}$ is abelian: 
The first condition of a commutative $\SL_2$-twin trees implies that
$U_{\alpha}$ is isomorphic to the additive group of a field, the second condition implies that 
$[U_{\alpha},U_{\beta}] =1$ if $\partial \alpha = \partial \beta$. Moreover, the group 
$U_{\epsilon}$ 
acts regularly on $\partial T_+ \setminus \{\epsilon\}$. Set $G:=\la 
U_{\infty},U_{\mathfrak{o}}\ra$. We say that \textbf{$T$ induces a Moufang set at infinity} if 
$G$ is a rank one group with unipotent subgroups $U_{\infty}$ and $U_{\mathfrak{o}}$ or, equivalently,
if $\mathbb{M}(T):=(\partial T_+, (U_{\infty}^g)_{g \in G})$ is a Moufang set.

The main result of \cite{GrMoufangTwinTrees} can be summarized as followed.  

\begin{theorem}\label{ClassificationTheorem}
Let $T$ be a commutative $\SL_2$-twin tree which induces a Moufang set at infinity. Then 
$\mathbb{M}(T)$ is isomorphic to the Moufang set $\mathbb{M}(J)$ associated with one of the quadratic Jordan division algebras $J$ appearing on the following list:
\begin{enumerate}
\item $J$ is the skewfield of skew Laurent series over a field, i.e. 
$$J=K(\!(t)\!)_{\theta}:= \big\{\sum_{i=n}^{\infty} a_i t^i \; | \; n \in \ZZ, a_i \in K \big\},$$ 
where $K$ is a field, 
$\theta \in \Aut K$ and $t$ is an indeterminate over $K$ with $ta =a^{\theta} t$ for all $a \in K$.
\item $J=\mathcal{H}(K(\!(t)\!)_{\theta}, *)$, where $*$ is the involution of $K(\!(t)\!)_{\theta}$ defined
by $t^* =t$ and $a^* = a^{\sigma}$ for an involutory automorphism $\sigma$ of $K$ such that $\sigma \theta \sigma = \theta^{-1}$.
\item $J=\{\sum_{i=n}^{\infty} a_i t^i \; | \; n \in \ZZ, a_{2i} \in E, a_{2i+1} \in F\} \subseteq K(\!(t)\!)$, 
where $K$ is a field of characteristic $2$ and $E$ and $F$ are subfields 
of $K$ containing all squares of 
$K$.
\item $J=\{\sum_{i=n}^{\infty} a_i t^i \; | \; n \in \ZZ, a_{2i} \in E, a_{2i+1} \in F\} 
\subseteq K(\!(t)\!)_{\theta}$, where  $E$ is a subfield of $K$ containing all squares, 
$\theta$ is an involutory automorphism of $K$ and $F$ is the 
fixed field of $\theta$. This Jordan algebra is contained in $\mathcal{H}(K(\!(t)\!)_{\theta},*)$ 
with $*$ as in (b) for $\sigma=1$.
\end{enumerate}
\end{theorem}

\begin{remark}\label{rem:Classif}
\renewcommand{\labelenumi}{(\arabic{enumi})} 
\begin{enumerate}
\item If in case (i) or (ii) the order of $\theta$ is finite, then the little projective group $G^{\dagger}$ 
of the Moufang set is an algebraic group. The groups in (i) are forms of $A_n$ and the groups in (ii) are forms of $C_2$ (see the description in the introduction above for more details). If the order
of $\theta$ is infinite, then $G^{\dagger}$ is a classical  group but not an algebraic group. The cases (iii) and (iv) lead to groups 
of mixed type. They do not occur when the field $K$ is finite.

\item If $T$ is locally finite, then the field $K$ must be finite, so either
case (i) or (ii) holds. In case (i) the automorphism $\theta$ must have finite order $n$, 
so the center of $K(\!(t)\!)_{\theta}$ is the local field $F(\!(t^n)\!)$, where $F = K^\theta$ the fixed field of $\theta$. Hence $K(\!(t)\!)_{\theta}$ 
is a division algebra of degree $n$ over its center. Since $\Aut K$ is cyclic for a finite field 
$K$, we must have $\theta^2=1$ and $\sigma \in \{1,\theta\}$ in case (ii). Hence if $\theta\ne 1$, 
then $K(\!(t)\!)_{\theta}$ is a quaternion algebra. If $\chara K \ne 2$, then $*$ is a non-standard 
involution.

\item We remark that if $J$ is any finite-dimensional quadratic Jordan division algebra over 
a local function field $K(\!(t)\!)$, 
then the Moufang set $\mathbb{M}(J)$ appears in our list. By \cite{McZ} a quadratic
Jordan division algebra is a division algebra, a Hermitian algebra, a Jordan algebra of Clifford 
type or an Albert algebra. There are no Albert division algebras over $K(\!(t)\!)$, and since 
every quadratic form in more than four indeterminates is isotropic over $K(\!(t)\!)$, every quadratic 
Jordan division algebra of Clifford type is a subalgebra of a quaternion division algebra. 
Since the Brauer group $Br(K(\!(t)\!)) $ is isomorphic to $ \QQ/\ZZ$ (see \cite{Serre}, XIII, Prop. 6), for every natural number $n\geq 1$ there are exactly 
$\varphi(n)$ pairwise non-isomorphic division algebras of degree $n$ with center $K(\!(t)\!)$.
These are the algebras $L(\!(u)\!)_{\theta}$ with $[L:K] =n$, $u^n=t$ and $Gal(L|K) =\la \theta \ra$.
Since $L(\!(u)\!)_{\theta} $ has an involution if and only if $\theta^2=1$, our list is exhaustive. 
\end{enumerate} 
\end{remark}

\renewcommand{\labelenumi}{(\roman{enumi})} 
\section{Boundary Moufang trees}

\subsection{The setup}

Let $T$ be a thick tree, i.e. a simplicial tree all of whose vertices have degree~$\geq 3$. For each vertex $x \in V(T)$, we denote by $x^\perp$ the set of vertices adjacent to, but different from, $x$. The set of ends of $T$ is denoted by $\partial T$. We let $(U_{\e})_{\e \in \partial T}$ be a collection of subgroups of $\Aut(T)$  such that $(\partial T,(U_{\e})_{\e \in \partial T})$ is a Moufang set. The Moufang condition implies that \textit{the root groups $U_{\e}$ all consist of elliptic automorphisms of $T$}, since a hyperbolic element in  $U_{\e}$ would be a non-trivial element fixing an end of $T$ different from $\e$, thus violating the condition that the $U_{\e}$-action on $\partial T \setminus \{\e\}$ is free. The little projective group of this Moufang set is denoted by $G$. By construction $G$ is a subgroup of $\Aut(T)$.  

Throughout, we fix  an ordered apartment $\Sigma =(x_n)_{n \in \ZZ}$  in $T$. The end  of $\Sigma$ represented by 
 the half-apartment $\{x_n \; | \; n \leq 0\}$ (resp.  $\{x_n \; | \; n \geq 0\}$) is denoted by $\infty$ (resp. $\mathfrak{o}$). 
   
Let $(U,+)$ be a group and  $\alpha: U \to U_{\infty}$ be an isomorphism. For each $n \in \ZZ$, we set $U_n=\alpha^{-1} (U_{\infty,x_n})$, where $U_{\infty,x_n}$ denotes the stabiliser of $x_n$ in the root group $U_\infty$. We have $U_{n+1} \leq U_n$ for all $n$. 

We remark that, for any  $e \in U^\#$, the given Moufang set $(\partial T, (U_{\e})_{\e \in \partial T})$  
 is isomorphic to $\mathbb  M(U,\tau)$, where  $\tau =\mu_e$. Given  $a \in U^{\#}$, we let 
 $$\beta_a=\alpha_{a\tau^{-1}}^{\tau}$$ 
 be the unique element in $U_{\mathfrak{o}}$  with $(\mathfrak{o})^{\alpha_a} = \infty^{\beta_a}$.

\begin{lemma} \label{lem:reflection}
Let $n,m \in \ZZ$. If $a \in U_n \setminus U_{n+1}$ and $x \in U_m \setminus U_{m+1}$, 
then $\mu_a$ maps $x_m$ to $x_{2n-m}$ and $x\mu_a \in U_{2n-m}  \setminus U_{2n-m+1}$.  
\end{lemma}

\begin{proof}
Let $\beta_a$ be the unique element of $U_{\mathfrak{o} }$ with $({\mathfrak{o}})^{\alpha_a}={\infty}^{\beta_a}$. Then 
$x_n$ is fixed by $\beta_a$ since $x_n$ is the unique common vertex common to the three lines  $(\infty, \mathfrak{o}), (\infty, \mathfrak{o}^{\alpha_a})$ 
and $(\mathfrak{o}, \infty^{\beta_a})$. Thus $x_n$ is fixed by $\mu_a =\beta_{a^{-1}} \alpha_a \beta_a$. Since
$\mu_a$ interchanges $\infty$ and $\mathfrak{o}$, it follows that $\mu_a$ acts on $\Sigma=(\mathfrak{o}, \infty)$ as the reflection through $x_n$. Therefore $\mu_a$ maps $x_m$ to $x_{2n-m}$. By the same argument $\beta_{x}$ fixes $x_m$. Since
$\beta_{x} = \alpha_{x\mu_a}^{\mu_a^{-1}}$ we have 
$$\alpha_{x\mu_a} =\beta_x^{\mu_a} \in 
(U_{\mathfrak{o},x_m} \setminus U_{\mathfrak{o},x_{m-1}})^{\mu_a} =U_{\infty,x_{2n-m}} \setminus U_{\infty,x_{2n-m+1}}$$
so that $x\mu_a \in U_{2n-m}  \setminus U_{2n-m+1}$ as desired.
\end{proof}  

Let $\HH$ be the Hua group of $\mathbb  M(U,\tau)$, so $\HH = G_{\mathfrak{o}, \infty}$. Moreover, we set 
$$\HH_n=\langle \mu_a \mu_b \; | \; a,b \in U_n \setminus U_{n+1}\rangle.$$
By Lemma~\ref{lem:reflection}, the group $\HH_n$ fixes $x_n$ and 
thus acts trivially on $\Sigma$. 

\begin{lemma}\label{normalization}
For all $m,n \in \ZZ$, the subgroups $\HH_m$ and $\HH_n$ normalize each other.
\end{lemma} 

\begin{proof}
We have observed that $\HH_n$ acts trivially on $\Sigma$, and thus fixes $x_m$. Since $\HH_n \leq \HH$ normalizes $U_{\mathfrak{o}}$ and 
$U_{\infty}$, it follows that $\HH_n$ normalizes the set $\{\mu_a \; | \; a \in U_m \setminus U_{m+1}\}$. Therefore
$\HH_n$ normalizes $\HH_m$.
\end{proof}

\subsection{The local Moufang sets}

An important point to analyze is when the set of neighbours $x^\perp$ of a vertex $x \in V(T)$ carries the structure of a Moufang set that is invariant under the action of the stabiliser $G_x$. Throughout this subsection, we assume that 
$$U_{n+1} \trianglelefteq U_n \hspace{1cm} \text{for all } n \in \ZZ.$$
We first show that this condition is indeed sufficient to ensure the existence of canonical Moufang sets localised at each vertex.  

\begin{lemma}\label{local Moufang sets} 
Let $x$ be a vertex of $T$. 
\begin{enumerate}
\item For any $\e \in \partial T$ and $z \in [x, \e)$ different from $x$, the group $U_{\e, x}$ acts trivially on $z^\perp$. 

\item 
Given a vertex $y \in x^\perp$ and  two ends $\e,\f \in \partial T$ such that $y \in [x,\e) \cap [x,\f)$, then $U_{\e,x}$ and $U_{\f,x}$ induce the same permutation groups on the set $x^{\perp}$.  This subgroup of $\Sym(x^\perp)$ is denoted by $U_y$.

\item 
The pair $(x^{\perp}, (U_y)_{y\in x^{\perp}})$ is a Moufang set, called the \textbf{local Moufang 
set} at $x$.
 \end{enumerate}
 \end{lemma}
 \begin{proof}
 (i) Since $U_{\e,x} \leq U_{\e,z}$ for all $z \in [x,\e)$, it suffices to prove the statement in the 
 case where $z$ is adjacent to $x$. The hypothesis that $U_{n+1} \trianglelefteq U_n$ for all $n \in 
 \ZZ$ implies that $U_{\e, x} \trianglelefteq U_{\e, z}$. Since $U_{\e, z}$ consists of elliptic 
 elements and acts transitively on $\partial T \setminus \{\e\}$, we infer that $U_{\e, z}$ is 
 transitive on the vertices in $z^\perp$ not belonging to the ray $[z, \e)$. This implies that 
 $U_{\e, x} \trianglelefteq U_{\e, z}$ acts trivially on $z^\perp$. 

\medskip \noindent  
 (ii) Let $\e^{\prime}$ be an end with $x \in [y,\e^{\prime})$. Then there is $g \in
  U_{\e^{\prime}}$ such that $U_{\e}^g =U_{\f}$. Since $g$ is elliptic and since $y \in (\e^{\prime},\e) \cap 
  (\e^{\prime},\f)$, it follows that   $g$ fixes $y$. From (a), it follows that  $g$ induces the identity on $x^{\perp}$,  so that $U_{\e,x}$ and $U_{\f,x} =U_{\e,x}^g $ indeed induce the same permutation groups on the set $x^{\perp}$.
  
\medskip \noindent 
  (iii) Let $y,z \in x^{\perp}$ be distinct and let $\e,\f$ be ends with $y\in [x,\e)$ and $z \in [x,\f)$. 
  The image of $U_{\e,x}$ in $\Sym(x^{\perp})$ is $U_y$, while the image of $U_{\f,x}$ is 
  $U_z$. Let $a \in U_{\e,x} \setminus U_{\e,z}$. Then there is $b\in U_{\f}$ with 
  $U_{\e}^b =U_{\f}^a$ and thus $\e^b =\f^a$. Since $b$ is elliptic and since $x \in (\e,\f) \cap 
  (\e^b,\f)$, we conclude again that $b$ fixes $x$. Thus $U_{\f,x}^a =(U_{\f} \cap G_x)^a = U_{\f}^a 
  \cap G_x^a =
  U_{\e}^b \cap G_x^b = (U_{\e} \cap G_x)^b = U_{\e,x}^b$ and so $U_y^b =U_z^a$.   
 \end{proof}

 \begin{lemma}\label{tau} 
Let $k \in \ZZ$, let $e \in U_k \setminus U_{k+1}$, and set $\tau=\mu_e$. Let also $m > n \in \ZZ$. Then for any  $a \in U_{n} \setminus U_{n+1}$ and  $b \in U_m$, there exists  $d\in U_{2k-2n+m+1}$ such that $(b+a)\tau =d + b\mu_a \tau + a \tau$.
 \end{lemma} 
 \begin{proof}
 Let $\rho =\mu_a$ and set $c = \beta_{b+a} \beta_a^{-1}$. Then we have 
 $$x_{m+1}^{\alpha_{b+a}} = x_{2n-m-1}^{\beta_{b+a}} =x_{2n-m-1}^{c\beta_a},$$ 
 so that
 $$x_{m+1}^{\alpha_b \alpha_a  \beta_a^{-1}}=x_{2n-m-1}^c.$$
Since $x_{m+1}^{\alpha_b}$ is adjacent to $x_m$, it is fixed by $\beta_{-a}$ in view of Lemma~\ref{local Moufang sets}(i), and we deduce
 $$x_{2n-m-1}^c = x_{m+1}^{\alpha_b \beta_{-a} \alpha_a \beta_a^{-1}} =x_{m+1}^{\alpha_b  \mu_a}
 =x_{2n-m-1}^{\mu_a^{-1} \alpha_b \mu_a},$$
 where the last equality follows from Lemma~\ref{lem:reflection}.
 Thus there exists an element $u \in U_{\mathfrak{o},2n-m-1}$ with $c = u \alpha_b^{\mu_a}$. Therefore  
 $$\alpha_{(b+a)\rho}^{\rho^{-1}} = \beta_{b+a}
  =c\beta_a=u\alpha_b^{\rho}\alpha_{a\rho}^{\rho^{-1}}.$$
Applying $\rho$ on both sides, we obtain
  $$\alpha_{(b+a)\rho} = u^{\rho} \alpha_b^{\rho^2} \alpha_{a\rho}= u^{\rho} \alpha_{b\rho^2} 
  \alpha_{a\rho} =  u^{\rho} \alpha_{b\rho^2 +a\rho}.$$
  Applying $\rho^{-1}\tau$ we get
$$\alpha_{(b+a)\tau} = u^\tau \alpha_{b\rho \tau + a \tau}.$$
Since $u^{\tau} \in U_{\mathfrak{o},2n-m-1}^{\tau} =U_{\infty, 2k -2n+m+1}$ by Lemma~\ref{lem:reflection}, there exists
  $d \in U_{m-2n+1}$ such that  $u^{\tau} = \alpha_d$, and we obtain $(b+a)\tau = d +b\rho \tau +a\tau$,  as required.
 \end{proof}

The following result provides important additional information on the local Moufang sets.

\begin{lemma}\label{local Moufang sets II}
Let $n \in \ZZ$ and $e \in U_n \setminus U_{ n+1}$ and set $\tau=\mu_e$. Let moreover 
$\UU =  U_{n}/U_{n+1}$ 
and $\ttau$ be the permutation of $\UU^\#$ defined by $x+ U_{ n+1}  \mapsto x\tau + U_{n+1}$.
Then $M(\UU,\ttau)$ is a Moufang set which is isomorphic to the local Moufang set at $x_n$.  
In particular, if the Moufang set $\mathbb  M(U,\tau)$ is special, then so are all  the local Moufang sets. 
\end{lemma}
 
\begin{proof}
Applying Lemma~\ref{tau} with $k = n$ and $m = n+1$, we see that the map $\ttau$ is well defined. The first   claim follows from Lemma~\ref{local Moufang sets}, while the second follows from the first.
\end{proof}

\subsection{Local action of the Hua group}
  
In the rest of the subsection, we will assume further that $U$ is abelian. In particular the condition $U_{n+1} \trianglelefteq U_n$ trivially holds for all $n \in \ZZ$. Notice that the Moufang set $\mathbb  M(U, \tau)$ is necessarily proper, since the Hua group $\HH$ contains hyperbolic elements: indeed, every element element of the form  $\mu_a \mu_b$, with $a \in U_0 \setminus U_1$ and $b\in U_1 \setminus U_2$, acts as a translation of length~$2$ on the apartment $\Sigma$. The main result from  \cite{S2} therefore ensures that $\mathbb  M(U,\tau)$ is special. This fact will be used extensively, without further notice.

Notice that, in the present setting, the quotient group $\HH_n/\HH_n^{(1)}$ is isomorphic to the Hua group of the local Moufang set at $x_n^{\perp}$.  A key question consists in understanding how the group $\HH_n$ acts on the local Moufang sets at $x_{n+1}$ and $x_{n-1}$. This question will be answered with the help of the following two technical but crucial results.

\begin{lemma}\label{Hua maps} 
Assume that $U$ is abelian. 
Let $i \in \{0,1\}$, let $e \in U_i\setminus U_{i+1}$ and set $\tau =\mu_e$. 
Given $a \in U_i \setminus U_{i+1}$ and $b \in U_m$ with $m >i$, the following holds for any $x \in U_i^\#$:
 $$xh_{a,b} \equiv   - bh_x h_{x\tau, a} \equiv   -bh_{x,a\tau} h_a \mod  U_{m+1}.$$ 
 \end{lemma}

\begin{proof}
We only prove this for $i=0$;  the proof is similar in case $i = 1$.

We recall that $\mathbb  M(U,\tau)$ is special. Moreover  $\tau = \tau^{-1}$ by Lemma~\ref{mua=mub}.  

 Let $x\in  U_0^\#$. By the definition of the Hua maps (see Definition 3.4 in \cite{DS1}), we have
 $$xh_{a+b} =((x\tau+a+b)\tau -(a+b)\tau)\tau +a+b.$$

Suppose first $x\tau +a \not\in U_1$. If $x \in U_n \setminus U_{n+1}$ with $n \geq 0$, then 
 $x\tau,x\tau +a \in U_{-n} \setminus U_{-n+1}$, so by Lemma~\ref{tau}
  $$(x\tau +a +b)\tau \equiv (x\tau +a)\tau +bh_{x\tau +a}^{-1} \mod  U_{m+2n+1}.$$
By Lemma~\ref{basic formulas}(a), we have
$$((x\tau +a)\tau -a\tau)\tau = x\tau \mu_a - a = xh_a -a.$$
If $xh_a -a \in U_1$, then 
$x\tau \mu_a \equiv a \mod U_1$, so $x\tau \equiv a\mu_a \equiv -a \mod  U_1$,  
contradicting the assumption that $x\tau +a \not\in U_1$. 
Thus $((x\tau +a)\tau -a\tau)\tau \in U_0 \setminus U_1$, so that $(x\tau +a)\tau -a\tau  \in U_0 \setminus U_1$ by Lemma~\ref{lem:reflection}. Now we have
\begin{align*}
(x\tau +a+b)\tau -(a+b)\tau  
&\equiv  (x\tau+a)\tau +bh_{x\tau +a}^{-1} -(a+b)\tau \\
&\equiv   (x\tau +a)\tau -a\tau +bh_{x\tau +a}^{-1}-bh_a^{-1} \mod  U_{m+1},
\end{align*}
where Lemma~\ref{tau} has been used to evaluate $(a+b)\tau$. 
Applying $\tau$ and using  Lemma~\ref{tau} once more, we obtain 
\begin{align*}
((x\tau +a+b)\tau -(a+b)\tau) \tau 
& \equiv ( (x\tau +a)\tau -a\tau +bh_{x\tau +a}^{-1} -bh_a^{-1})\tau \\
&\equiv  
((x\tau +a)\tau -a\tau)\tau +bh_{x\tau +a}^{-1} h_{(x\tau +a)\tau -a\tau}^{-1} 
-bh_a^{-1} h_{(x\tau +a)\tau -a\tau}^{-1}
\end{align*}
modulo  $ U_{m+1}$.
Applying Lemma~\ref{basic formulas}(b), we deduce
\begin{align*}
h_{x\tau+a}^{-1} h_{(x\tau+a)\tau -a\tau}^{-1} 
&=  \mu_{x\tau +a} \tau \mu_{(x\tau+a)\tau -a\tau} \tau 
\  = \ \mu_{x\tau +a} \mu_{((x\tau+a)\tau -a\tau)\tau} \\
&=  \mu_{x\tau+a}\mu_{x\tau+a} \mu_{x \tau} \mu_a
\ = \ \tau \mu_x \tau \mu_a \\
&=  h_x h_a
\end{align*}
and
\begin{align*}
h_a^{-1} h_{(x\tau+a)\tau -a\tau}^{-1}  
&=  \mu_a \tau \mu_{(x\tau+a)\tau -a\tau}\tau 
\ = \ \mu_a \mu_{((x\tau+a)\tau -a\tau)\tau} \\
&=  \mu_a \mu_{-a} \mu_{-x\tau}  \mu_{x\tau+a}
\ = \ \mu_{x\tau} \mu_{x\tau +a} 
\ = \   \tau \mu_x \tau \mu_{x\tau +a} \\
&=  h_x h_{x\tau+a}.
\end{align*}
Moreover we have $h_{x\tau} = \tau \mu_{x\tau} = \tau \tau^{-1} \mu_{-x} \tau = h_x^{-1}$. 
Thus
\begin{align*}
xh_{a+b} &=  ((x\tau +a+b)\tau -(a+b)\tau)\tau +a+b \\
&\equiv  ((x\tau +a)\tau -a\tau)\tau + a +b h_x h_a - b h_x h_{x\tau +a}  + b\\
&\equiv  xh_a  +b h_x h_a - b h_x h_{x\tau +a}  + b h_x h_{x\tau} \\
&\equiv  xh_a -bh_x(h_{x\tau +a}-h_a -h_{x\tau}) \\
& \equiv   xh_a -bh_x h_{x\tau,a} \mod  U_{m+1}.
\end{align*}
Since $xh_b \in U_{2m} \leq U_{m+1}$, we have 
$$x h_{a,b} = xh_{a+b} - xh_a -xh_b \equiv -bh_x h_{x\tau,a} \mod U_{m+1},$$
which confirms  the first equality asserted by the lemma.  
Proposition~5.8(2) from \cite{DS3} ensures that $h_x h_{x,a\tau} h_a^{-1} = h_x h_{x\tau, a} h_{a\tau} = h_{x,a\tau} $, and the second equality from the lemma follows. 

\medskip
We now suppose that $x\tau +a \in U_1$. Then Lemma~\ref{tau} ensures 
$$x=x\tau^2 = (-a+x\tau +a)\tau \equiv (-a)\tau +(x\tau +a) h_a^{-1} \mod  U_1,$$
 so that $c:=x+a\tau=x-(-a)\tau \in U_1$. Therefore
$$x h_{a,b} = (-a\tau +c) h_{a,b} = 2b +ch_{a,b}$$ 
by Lemma~\ref{hab}, and 
$$bh_{a\tau,x} h_a = bh_{a\tau, c-a\tau} h_a.$$
Now Lemma~\ref{biadditivity} implies 
$$h_{a\tau,c-a\tau} =h_{0,c} -h_{a\tau,c}-h_{-a\tau,c}+h_{a\tau,c} +h_{a\tau,-a\tau}=-h_{-a\tau,c}-2h_{a\tau} = -h_{-a\tau,c} -2h_a^{-1}.$$
Since $-c \in U_1$, we have 
$-c\tau \in U_{-k} \setminus U_{-k+1}$ for some $k>0$ by Lemma~\ref{lem:reflection}, so that $a - c\tau 
\not\in U_1$. We may therefore apply the first part of the proof to $-c$ and deduce that
$- c h_{a,b} \equiv -bh_{-c, a\tau} h_a \mod U_{m+1}$.  Since $\mu_d = \mu_{-d}$ for all $d \in U^\#$ by 
Lemma~\ref{mua=mub}, we also have $h_d = h_{-d}$ and $h_{d, f} = h_{-d, -f}$. 
Thus we get
\begin{align*}
bh_{a\tau,x} h_a 
&=  -bh_{-a\tau,c} h_a -2b  
\ = \  -bh_{a\tau,-c} h_a -2b \\
&\equiv    -c h_{a,b} -2b \\
&\equiv   -xh_{a,b} \mod  U_{m+1}.
\end{align*}
As in the first case, this also implies that $x h_{a, b} \equiv -b h_x h_{x\tau, a} \mod U_{m+1}$. 
\end{proof}

 \begin{prop}\label{Hua maps II}
 Assume that $U$ is abelian.  Let $i \in \{0,1\}$, let $e \in U_i \setminus U_{i+1}$ and set $\tau =\mu_e$. For any $a \in U_i \setminus U_{i+1}$ and $b \in U_m$ with $m>i$, we have the following:
 \begin{enumerate}
 \item $x h_a \equiv x h_{a+b} \mod U_m$ for all $x \in U_i$.
 
 \item $x h_a \equiv x h_{a+b} \mod U_{m+1}$ for all $x \in U_m$.
\end{enumerate}
 \end{prop}
 
 \begin{proof}

 \medskip \noindent
 (i) Let $x \in U_i^\#$. By Lemma~\ref{Hua maps}, we have $xh_{a+b} - x h_a = x h_{a, b} +x h_b \equiv -b h_{x, a\tau} h_{a} + xh_b \mod U_{m+1}$. Since $xh_b$ and $-b h_{x, a\tau} h_{a} $ are both   contained in $U_m$, we deduce that   $xh_{a+b} \equiv x h_a \mod U_m$ as desired.

 \medskip \noindent
 (ii) Let $x \in U_m$ and $y \in U_i \setminus U_{i+1}$. 
Lemma~\ref{Hua maps} implies that 
 $$yh_{a,b+x} \equiv yh_{a,b} +yh_{a,x} \mod  U_{m+1}.$$
Moreover $yh_{b,x} =yh_{b+x} - yh_b - yh_x \in U_{2m} \leq U_{m+1}$. Therefore, Lemma~\ref{biadditivity} implies
\begin{align*}
yh_{a+b,x} &=  y h_{a, b+x} - y h_{a, b} + y h_{b, x}\\
& \equiv  y h_{a,x}  \mod  U_{m+1}.
\end{align*}
Applying Lemma~\ref{Hua maps} again twice, we obtain
 $$-xh_y h_{y\tau,a} \equiv yh_{a,x} \equiv yh_{a+b,x} \equiv -x h_y h_{y\tau,a+b} 
 \mod  U_{m+1}.$$
Lemma~\ref{lem:reflection} ensures that $h_y$ acts trivially on the apartment $\Sigma$, and therefore normalises $U_m$. We may thus replace $x$ by $xh_y^{-1}$ in the above equation, which yields
 $$xh_{y\tau,a} \equiv xh_{y\tau,a+b} \mod  U_{m+1}.$$ 
%
%
%
Replacing $y\tau$ by $y$, we obtain $x h_{y, a} \equiv x h_{y, a+b} \mod U_{m+1}$. Specialising to $y = -a$, we get 
$-2xh_a = xh_{a,-a} \equiv xh_{-a,a+b} \equiv xh_b -xh_{-a} -xh_{a+b} \mod U_{m+1}$ and thus
$-x h_{a} \equiv x h_b - x h_{a + b} \equiv -x h_{a +b} \mod U_{m+1}$ since $x h_b \in U_{3m} \leq U_{m+1}$. 
\end{proof}

\section{Locally finite abelian boundary Moufang trees}

We shall now focus on the case when  $T$ is locally finite and $U$ is abelian. 

\subsection{The local Moufang sets are finite projective lines}

\begin{lemma}\label{local situation} 
Assume that $T$ is locally finite and that $U$ is abelian. Then for 
 every vertex $x$, there is a 
 finite field $\mathbf{F}$ such that the local Moufang set at $x$ is isomorphic to 
 $\mathbb {M}(\mathbf{F})$.
 \end{lemma}
 \begin{proof}
As observed above, the Moufang set $\mathbb  M(U,\tau)$ is special as a consequence of \cite{S2}. By Lemma~\ref{local Moufang sets II} the local Moufang sets are also special,  
so the claim follows since every finite special Moufang set is of the form $\mathbb {M}(\mathbf{F})$ for some finite field $\mathbf F$, in view of \cite{S1} and \cite{DS2}.
\end{proof}  
 
Recall that the little projective group $G$ is $2$-transitive on the set of ends $\partial T$. This implies that $G$ is edge-transitive on $T$ (see Lemma~3.1.1 from \cite{BurgerMozes}). Moreover $G$ is generated by elliptic automorphisms of $T$, so that $G$ preserves the canonical bipartition of the vertex set. We conclude that $G$ acts with exactly two orbits of vertices. From Lemma~\ref{local situation}, it follows that there are prime powers $q_0$ and $q_1$ such that  the local Moufang sets are all isomorphic to $\mathbb {M}(\FF_{q_0})$ or  $\mathbb {M}(\FF_{q_1})$.    


\subsection{Construction of finite root subgroups}

We will now assume further that $U$ is of exponent $p$ for some prime $p$. This implies that the root groups of the local Moufang sets are also groups of exponent $p$. Therefore the finite fields $\FF_{q_0}$ and $\FF_{q_1}$ are both of characteristic $p$ in that case. 

Our next goal is to construct a finite root subgroup $V_n \leq U$ fixing the vertex $x_n$ and whose image in $\Sym(x_n^\perp)$ is a root group of the local Moufang set at $x_n$. This will be achieved in Proposition~\ref{prop:FiniteRootSubgroup} below. We first need to collect a number of technical preparations.


\begin{lemma}\label{structure of the Hua group}
Assume that $T$ is locally finite and that $U$ is abelian of exponent $p$ for some prime $p$. 
Let $i \in \{0,1\}$ and $n \geq i$. Let $V \leq U_i$ be a subgroup satisfying the following conditions:
\begin{itemize}
\item[(1)] $U_i =V +U_{i+1}$ if $n>i$, and $V=U_i$ if $n=i$.
\item[(2)] $V \cap U_{i+1} =U_{n+1} $.
\item[(3)] $V$ is normalized by $\tilde{\cal H}:= \langle \mu_a \mu_b \; | \; a,b \in V^* \rangle$, where 
  $V^* =  V \setminus U_{n+1}$.
\end{itemize}
Let $\tilde{\mathcal H}_0 = C_{\tilde{\cal H}}(V/U_{n+1})$, let  ${\cal H}  = \tilde{\cal H}/C_{\tilde{\cal H}}(V/U_{n+2})$ and let 
${\cal H}_0$ be the image of $\tilde{\mathcal H}_0$ in $\mathcal H$.
Then the following assertions hold. 
\begin{enumerate}
\item $[\tilde{\mathcal H}_0, U_{n+1}] \leq U_{n+2}$. 
\item ${\cal H}_0$ is a normal elementary-abelian $p$-subgroup of ${\cal H}$.
\item ${\cal H}/{\cal H}_0$ is cyclic of order $q_i-1$ if $p=2$ (resp. $\frac{q_i-1}{2}$ if $p \ne 2$).
\item For all $e \in V^*$ there is $a \in V^*$ such that the image of $\la \mu_e \mu_a \ra$ in $\mathcal H$ is a complement for $\mathcal H_0$.
\end{enumerate}     
\end{lemma}

\begin{proof}
(i) 
Set $N_1 = \tilde{\mathcal H}_0 =  C_{\tilde{\cal H}}(V/U_{n+1})$ and $N_2 =C_{\tilde{\cal H}}(U_{n+1}/U_{n+2})$.  
Let $e \in V^*$. For $a \in V^*$  set $h_a =\mu_e \mu_a$. Note that $\tilde{\cal H}$ is generated by 
$\{h_a; a\in V^*\}$ since $\mu_a \mu_b = h_a^{-1} h_b$ for all $a,b \in V^*$. 
By Proposition~\ref{Hua maps II} we have $h_a^{-1} h_{a+b} \in N_1 \cap N_2$ for all
$a \in V^*$ and $b\in U_{n+1}$. 
In particular, given a subset  $X \subset V$ such that 
$\{0\} \cup X$ is a system of coset representatives for $U_{n+1}$ in $V$, then the quotient $\tilde{\mathcal H}/ N_1 \cap N_2$ is generated by the images of $\{h_x \; | \; x \in X\}$.

By hypothesis,  the map
$\varphi: V/U_{n+1} \to U_i/U_{i+1}:a+U_{n+1} \mapsto a+U_{i+1}$ is an isomorphism which is $\tilde{\cal 
H}$-equivariant. It follows that
$\tilde{\cal H}/N_1$ is cyclic of order $q_i-1$ if $p=2$ (resp. $\frac{q_i-1}{2}$ if $p$ is odd).
Thus there is $a \in V^*$ such that $\tilde{\cal H}/N_1$ is generated by $h_a N_1 $. For $m$ even, 
we have
\begin{align*}
h_a^m &=  (\mu_e \mu_a)^m 
\ = \  \mu_e \underbrace{\mu_a \mu_e \ldots \mu_a}_{m-1} \mu_e 
\underbrace{\mu_a \mu_e \ldots
 \mu_a}_{m-1} 
 \ = \ \mu_e \underbrace{\mu_a \mu_e \ldots \mu_a}_{m-1} \mu_e \mu_e \mu_e 
 \underbrace{\mu_a \mu_e \ldots
 \mu_a}_{m-1} \\
 &=     \mu_e h_a^{-\frac{m}{2}} \mu_e h_a^{\frac{m}{2}}
 \ = \  \mu_e \mu_{eh_a^{\frac{m}{2}}} 
 \ = \ h_{eh_a^{\frac{m}{2}}}.
\end{align*}
Similarly, for $m$ odd, we have
\begin{align*}
h_a^m &=   (\mu_e \mu_a)^m 
\ = \  \mu_e \underbrace{\mu_a \mu_e \ldots \mu_e}_{m-1} \mu_a 
\underbrace{\mu_e \ldots \mu_a}_{m-1} \\
& =    \mu_e h_a^{-\frac{m-1}{2}} \mu_a h_a^{\frac{m-1}{2}} 
\ = \  \mu_e \mu_{ah_a^{\frac{m-1}{2}}} 
\  = \ h_{ah_a^{\frac{m-1}{2}}}.
\end{align*}
From Lemma~\ref{local situation}, it follows that $U_i^*/U_{i+1} = (a+U_{i+1}) \tilde{\cal H} \cup 
(e+U_{i+1}) \tilde{\cal H}$. Transforming by $\varphi$, we deduce that $V^*/U_{n+1}= (a+U_{n+1}) 
\tilde{\cal H} \cup (e+U_{n+1}) \tilde{\cal          H}$. 
Now we set $m = | \tilde{\mathcal H}/N_1|$ (so $m=q_i-1$ if $p=2$ and $m=\frac{q_i-1}{2}$ if $p$ is odd) 
and $X= \{eh_a^j; j = 0, \ldots, m-1\} \cup \{ah_a^j; j=0, \ldots, m-1\}$. Hence $X \cup \{0\}$ is a system 
of coset representatives for $U_{n+1}$ in $V$. We have seen that 
$\tilde{\mathcal H}/ N_1 \cap N_2$ is generated by the images of $\{h_x \; | \; x \in X\}$, and we deduce  that  $\tilde{\mathcal H}/ N_1 \cap N_2$ is cyclic and generated by $h_a$. 
Moreover, since $h_a^m \in N_1$, by~\ref{mua=mub} there is $b \in U_{n+1}$ such that 
        $eh_a^{\frac{m}{2}} \in \{e+b,-e+b\}$ for $m$ even and $ah_a^{\frac{m -1} 2} \in \{e+b,-e+b\}$ for $m$ odd. Thus 
        $h_a^m  = h_{e+b}$ or $h_a^m  = h_{-e+b}$, and we deduce from Proposition~\ref{Hua maps II} that $h_a^m \in N_2$.     
Therefore, we have $N_1 \leq N_2$. This proves (i). 

\medskip \noindent
(ii) Clearly $\tilde{\mathcal H}_0$ is normal in $\tilde{\mathcal H}$, so that   ${\cal H}_0$ is normal in ${\cal H}$. Moreover, since 
        ${\cal H}_0$ centralizes $U_{n+1}/U_{n+2}$ by Part (i), the commutator map
        $$V/U_{n+1} \times {\cal H}_0 \to U_{n+1}/U_{n+2}:(a+U_{n+1},h) \mapsto [a,h] 
        +U_{n+2}$$ 
        is a homomorphism. Since the centraliser  $C_{{\cal H}_0}(V/U_{n+2})$   is trivial, this induces an injective homomorphism of $\mathcal H_0$ into the additive group of homomophisms from $ V/U_{n+1}$ to $U_{n+1}/U_{n+2}$. It follows  ${\cal H}_0$ is 
        an elementary-abelian         $p$-group.
        
\medskip \noindent
(iii)
With the notation of Part (i), we have ${\cal H}/{\cal H}_0 \cong \tilde{\cal 
        H}/N_1$, which is cyclic of 
        order $q_i-1$ for $p=2$ and $\frac{q_i-1}{2}$ if $p$ is even. 
        
\medskip \noindent
(iv) 
We have $h_a^m \in N_1 \cap N_2$, thus the image of $h_a^m$ in ${\cal H}$ lies in 
        ${\cal H}_0$. Hence $h_a^m C_{\tilde{\mathcal H}}(V/U_{n+2})$ has order $1$ or $p$ by (ii). In the latter 
        case we replace 
        $a$ by $eh_a$ if $p=2$, and by $ah_a^{\frac{p-1}{2}}$ if $p$ is odd. Now $h_a 
        C_{\tilde{\mathcal H}}(V/U_{n+2})$ has order $m$ and hence it therefore generates a complement for 
        ${\cal H}_0$ in ${\cal H}$.
\end{proof}

It is important to control the local action of the subgroup $\tilde{\mathcal H} \leq \HH$ afforded by Lemma~\ref{structure of the Hua group} on the local Moufang sets at $x_n$ and $x_{n+1}$. This is achieved by the next lemma, using a coordinatization of those local Moufang sets by the finite fields $\FF_{q_0}$ and $\FF_{q_1}$. 

\begin{lemma}\label{quadratic} 
Retain the notation and hypotheses of Lemma~\ref{structure of the Hua group}. 
Set $\FF=V/U_{n+1}$ and $\mathbf{E}=U_{n+1}/U_{n+2}$. 
Fix $e \in V^*$, and for each $a \in V^{\#}$, set  $h_a = \mu_e \mu_a$. Let also  $\varphi_a$ (resp. $\psi_a$) be the image of $h_a$ in  $End(\FF)$ (resp. $End(\mathbf{E})$), and define $\varphi_0$ (resp. $\psi_0$) as the   zero-map of $\FF$ (resp. $\mathbf{E}$). 
Then we have the following. 
\begin{enumerate}
\item $\varphi_a$ and $\psi_a$ depend only on the coset $a+U_{n+1}$. We may thus view $\varphi_a$ and $\psi_a$ as functions of $a \in \FF$. 

\item There is a multiplication $\cdot$ on $\FF$ such that $(\FF,+,\cdot)$ is a finite field 
with neutral element $e +U_{n+1}$, $a\mu_e +U_{n+1} = -(a +U_{n+1})^{-1}$ for all $a \in 
V^*$ and
$x\varphi_y =x \cdot y^2$ for all $x,y \in \FF$. Moreover, $\varphi_{x \cdot y} =
\varphi_x \varphi_y$ 
and $\psi_{x \cdot y} =\psi_x \psi_y$.

\item The map $\psi \colon \FF \to End(\mathbf E) : a\mapsto \psi_a$ is a multiplicative quadratic map.

\item There is a multiplication $\cdot$ on $\mathbf{E}$ such that $(\mathbf{E},+,\cdot)$ is a finite field
and such that there is an embedding $\iota$ from $\FF$ to the algebraic closure of $\mathbf{E}$ and 
a natural number $0 \leq s \leq \frac{r_i}{2}$ with 
$x\psi_y = x \cdot \iota(y)^{1 +p^s}$ for all $x \in \mathbf{E}$ and $y \in \FF$, where $r_i = \log_p q_i$.
\end{enumerate}
\end{lemma}

\begin{proof}
(i) 
If $  \in V^*$, then Propostion~\ref{Hua maps II} ensures that $\varphi_a =\varphi_{a+b}$ for any $b \in U_{n+1}$. By 
Lemma~\ref{structure of the Hua group}(i), this also implies that   $\psi_a =\psi_{a+b}$. If $a\in U_{n+1}^\# $,
 then $U_m h_a \leq U_{m+2n+2-2i}$ for all $m \in \ZZ$, so 
$\varphi_a =0$ and $\psi_a =0$.

\medskip \noindent
(ii) 
Since the local Moufang set at $x_i$ is isomorphic to $\mathbb {M}(\FF_{q_i})$, we have
a multiplication $\cdot$ on $U_i/U_{i+1}$ with $(a+U_{i+1})\cdot (b +U_{i+1})^2 = 
  a h_b +U_{i+1}$ for all $a,b \in U_i$. Since $U_i/U_{i+1}$ and $\FF$ are isomorphic 
  $\tilde{\cal H}$-modules with $\tilde{\cal H}$ as in Lemma~\ref{structure of the Hua group}, 
  we get a multiplication $\cdot$ on $\FF$ with $a\varphi_b = a\cdot b^2$ for all
  $a,b \in \FF$. Hence $\FF \cong \FF_{q_i}$, and we have $\varphi_{a \cdot b} = \varphi_a \varphi_b$. 
  Since $a\mu_e =-a\mu_a \mu_e = -ah_a^{-1}$, we get
  $a\mu_e +U_{n+1} =-a h_a^{-1}+U_{n+1}=-(a+U_{n+1})^{-1}$ for all $a \in V^*$.
  
If $x,y,z \in V $ with $(x+U_{n+1}) \cdot (y+U_{n+1}) = z+U_{n+1}$, then 
  $h_x h_y h_z^{-1} \in C_{\tilde{\cal H} }(V/U_{n+1}) \leq C_{\tilde{\cal H}} (U_{n+1}/U_{n+2})$ 
  by Lemma~\ref{structure of the Hua group}(i). It follows that $\psi_{a\cdot b} =\psi_a \psi_b$ 
  for all $a,b \in \FF$.
  
 \medskip \noindent
(iii)  
%
The map $\psi$ is multiplicative by Part~(ii). Moreover the map $f(a, b) = \psi_{a+b} - \psi_a -\psi_b$ is biadditive by Lemma~\ref{Hua maps}. For any $n \in \ZZ$, we also have that $\psi_{n \cdot e} = n^2 \psi_e$ by Proposition~4.6(6) from \cite{DS3}. The claim follows.  

\medskip \noindent
(iv)  By Lemma~\ref{local situation}, the elementary abelian group $\EE$ carries a multiplication $\cdot$ that turns it into a field (isomorphic to $\FF_{q_0} $ or $\FF_{q_1})$. We let  $\EE' $ be the subring of   $End(\mathbf E)$ defined by $\EE' = \{R_y \; | \; y \in \EE\}$, where $R_y \in End(\EE)$  denotes the right-multiplication by $y$. Thus $\EE'$ is a field which is isomorphic to $\EE$. 

We claim that $\psi \colon \FF \to End(\mathbf E)$ takes values in $\EE'$, unless   $\EE$ is isomorphic to $\FF_9$. 

To establish the claim, we consider  the little projective group of the local Moufang set at $x_{n+1}$, which we denote by $\overline{G}$. 
Hence $\overline{G} \cong \mathrm{PSL}_2(\EE)$. For each $a \in \FF$, we have $h_a \in \tilde{\mathcal H} \leq \HH_i$, with the notation of Lemma~\ref{structure of the Hua group}. The group $\HH_i$ is contained in the Hua group $\HH$, and therefore normalises both $U_\infty$ and $U_{\mathfrak{o}}$; moreover $\HH_i$ fixes the apartment $\Sigma$ pointwise. This implies that the image of $h_a$ in $\Sym(x_{n+1}^\perp)$ normalises a pair of opposite local roots groups at $x_{n+1}$. Therefore $h_a$ is contained in the normaliser $\Gamma = N_{\Sym(x_{n+1}^\perp)}(\overline G)$, which is isomorphic to $\mathrm P\Gamma \mathrm L_2(\EE)$. In fact we see moreover that $h_a$ is normalises the Hua subgroup $\overline H$ of $\overline G$ in $\Gamma$. 

Keeping the notation of Lemma~\ref{structure of the Hua group}, we see from that lemma that $ \tilde{\mathcal H}/\tilde{\mathcal H}_0$ is isomorphic to the Hua group of the local Moufang set at $x_i$, and is cyclic. Moreover Lemma~\ref{structure of the Hua group} ensures that $\tilde{\mathcal H}_0$ acts trivially on $U_{n+1}/U_{n+2}$, which implies  that $\tilde{\mathcal H}_0$ acts trivially on $x_{n+1}^\perp$. Therefore, the image of $\tilde{\mathcal H}$ in $\Sym(x_{n+1}^\perp)$, which we denote by $Y$, is cyclic. 

We conclude that $Y$ is a cyclic subgroup of $N_\Gamma(\overline H) \cong \mathbf{E} \rtimes \Aut(\mathbf{E})$. 
We now invoke Proposition~\ref{Hua maps II}, which ensures that the respective images of $\tilde{\mathcal H}$ and $\HH_i$ in $\Sym(x_{n+1}^\perp)$  coincide. Since the group $\HH_i $ and $\HH_{n+1}$ normalize each other by Lemma~\ref{normalization}, we infer that the subgroups $Y$ and    $(\mathbf{E}^*)^2$ normalize each other. Since $Y$ is abelian, it follows that $Y$ centralises every  element of $\Gamma$ which is a commutator of an element of $Y$ and an element of $(\mathbf{E}^*)^2$. In particular, any $h \in \tilde{\mathcal H}$ induces a Galois automorphism $\sigma$ of $\mathbf{E}$ such that 
$\sigma$ fixes $y^{-1} y^{\sigma}$ for all $y \in (\mathbf{E}^*)^2$. By Lemma~\ref{Galois automorphism} it follows 
that $\sigma$ is the identity unless $\EE \cong \FF_9$. This means precisely that if  $\EE \not \cong \FF_9$, then every element $h \in \tilde{\mathcal H}$ acts on $\EE$ as an element of $\EE'$.  The claim stands proven. 

\medskip
Assume now that $\psi$ does not take all its values in $\EE'$. Then $\EE \cong \FF_9$ by the claim. We then invoke Lemma~\ref{F9}, whose hypotheses are satisfied in view of Part~(iii) and Lemma~\ref{normalization}. This implies that $\FF \cong \FF_9$, and that $\{\psi_a \; | \; a \in \FF\}$ generates a field $\KK \subset End(\EE)$ isomorphic to $\FF_9$. Let $1_\EE$ denote the unit element of the field $\EE$, and consider the map 
$$f \colon \KK \to \EE : \alpha \mapsto 1_\EE \alpha.$$
The image $f(\KK)$ is an additive subgroup of $\EE$ containing $1_\EE$. The group of invertible 
endomorphisms of $\EE$ that preserve the cyclic subgroup $\la 1_\EE \ra$ is of order~$12$. It follows 
that the multiplicative group $\KK^*$, which has order~$8$, contains an element that does not 
preserve $\la 1_\EE \ra$. Therefore $f(\KK)$ is not contained in $\la 1_\EE \ra$, so that $f$ is 
surjective. Hence $f$ is bijective. Notice moreover that $f$ induces an isomorphism of additive 
groups from $\KK$ to $\EE$. We now endow  $\EE$ with a new multiplication, which is the unique 
multiplication that makes $f$ an isomorphism of fields. With respect to this new multiplication, 
which will again be denoted by the same symbol $\cdot$, we see that the subring $\EE'' = \{R_y \; | 
\; y \in \EE\}  \subset End(\EE)$ coincides with $\KK$. Moreover, the multiplicative quadratic map  
$\psi$ takes values in $\EE''$ by construction. 

\medskip
At this point, we conclude that there are only two possibilities: either the multiplicative quadratic map $\psi$ takes values in the field $\EE'$, which is isomorphic to $\EE$, or in the field $\EE''$, which is also isomorphic to $\EE$. Moreover, we have canonical isomorphisms $\EE' \to \EE$ and $\EE'' \to \EE$ given by the map $R_y \mapsto y$. We shall now treat both cases simultaneously, by identifying $\EE'$ (resp. $\EE''$) with $\EE$ via that map. With this identification, we obtain in either case that $x \psi_y = x \cdot \psi_y$ for all $x \in \EE$ and $y \in \FF$, where $\cdot$ denote the multiplication of $\EE$, as constructed above in the two different cases respectively. 

We now apply Theorem~\ref{thm:MQM} to the multiplicative quadratic map $\psi \colon \FF \to \EE$. This yields an embedding $\iota:\FF \to \overline{\mathbf{E}}$ and two natural numbers $1 \leq k,l \leq r_i$ such that 
  $\psi_x$ coincides with the right multiplication by $\iota(x^{p^k+p^l})$ for all $x \in \FF$. (We recall from (ii) that $\FF$ has order $q_i = p^{r_i}$.) If $l-k \leq \frac{r_i}{2}$, we replace 
  $\iota$ by $\iota^{\prime} \colon \FF \to \overline{\mathbf{E}}: a\mapsto \iota(a^{p^{r_i-k}})$. If 
  $l-k > \frac{r_i}{2}$, we replace $\iota$ by $\iota^{\prime}\colon  \FF \to \overline{\mathbf{E}}:
  a \mapsto \iota(a^{p^{r_i-l}})$. In both cases we conclude that $\psi_x$ coincides with the right multiplication by $\iota(x)^{1 +p^s}$ with $0 \leq s \leq \frac{r_i}{2}$.
\end{proof}

 \begin{prop}\label{construction of root groups}
 Assume that $T$ is locally finite and that $U$ is abelian of exponent $p$ for some prime $p$. 
 Let $X$ be a finite solvable $p^{\prime}$-subgroup of the Hua group $\HH$, and let $i \in \{0,1\}$. 
 Then there is a descending chain of subgroups $U_i = V_i \geq V_{i+1} \geq \ldots$  satisfying the following conditions for all $n \geq i$. 
 \begin{enumerate}
 \item $V_n = V_{n+1} +U_{n+1}$.
 
 \item $U_{n+1} \cap V_{n+1} =U_{n+2}$.
 
 \item There is $e_n \in V_n \setminus U_{n+1}$ such that $a \mu_{e_n} \in V_n$ for all $a \in 
 V_n^* :=V_n \setminus U_{n+1}$. 
 
 \item  $V_n$ is invariant under $X$ and $H_n=\langle  \mu_a \mu_b \; | \; a,b \in V_n^* \rangle$. 

\end{enumerate}
\end{prop}
 
\begin{proof}
We work by induction on $n$. For the base case, we remark that  $V_i =U_i$ satisfies trivially (i) and (ii), while (iii) and (iv) are ensured by  Lemma~\ref{lem:reflection}. 

Assume now that we   have constructed $V_i \geq \dots \geq V_n$ for $n\geq i$ with the requested properties. Set $\overline{X}= X/C_X(V_n/U_{n+2})$, 
 ${\cal H} = H_n/C_{H_n}(V_n/U_{n+2})$ and 
 ${\cal H}_0 =  C_{\cal H}(V_n/U_{n+1})$. Since $X$ normalizes $V_n$, it follows that $\overline{X}$, regarded as 
 a subgroup of $\Aut(V_n/U_{n+2})$,  normalizes ${\cal H}$. We form the semi-direct product $\mathcal H \rtimes \overline X$.  
 
 Applying Lemma~\ref{structure of the Hua group} with 
 $V=V_n$,  we see  that ${\cal H}_0$ is a normal elementary-abelian $p$-group of 
 ${\cal H}$ and ${\cal H}/{\cal H}_0$ is cyclic of order $q_i-1$ if $p=2$ (resp.  
 $\frac{q_i-1}{2}$ if $p$ is odd). Therefore $\mathcal H \rtimes\overline X$ is a finite solvable group. By Hall's Theorem, the $p'$-subgroup $X$ is contained in a Hall $p'$-subgroup, which is a complement for $\mathcal H_0$ in $\mathcal H \rtimes\overline X$.  This implies that there is a complement  $\overline{H}$   of ${\cal H}_0$ in  ${\cal H}$ which is normalised by $\overline X$. 
 Let $H$ be the preimage of $\overline{H}$ in 
 $H_n$. 
 
 Since $\overline{ H}\, \overline{X}$ is a $p^{\prime}$-subgroup of $\Aut(V_n/U_{n+2})$ 
 which stabilizes $U_{n+1}/U_{n+2}$, we deduce from Maschke's theorem that there is an $HX$-invariant subgroup $V_{n+1}$ of $V_n$ 
 with $V_n =V_{n+1} + U_{n+1}$ and $U_{n+1} \cap V_{n+1} = U_{n+2}$. 
 
 It remains to check the property (iii). Indeed, if we find   $e \in V_{n+1} \setminus U_{n+2}$ satisfying  $b\tau \in V_{n+1} \setminus U_{n+2}$ for all $b \in 
 V_{n+1} \setminus U_{n+2}$, where $\tau = \mu_e$, then  we also have
 $bh_c = ((b\tau +c)\tau -c\tau)\tau +c \in V_{n+1} \setminus U_{n+2}$ for 
 all $b,c \in V_{n+1} \setminus U_{n+2}$ (see Definition~3.4 in \cite{DS1}). Hence (iv) follows from (iii).

Since all the complements of ${\cal H}_0$ in ${\cal H}$ are conjugate, we deduce from Lemma~
\ref{structure of the Hua group}(iv) there exist elements $a,e \in V_{n+1} \setminus U_{n+2}$ 
and $x,y \in U_{n+1}$ such that $\overline{H}$ is generated by the image
of $h= \mu_{e+x} \mu_{a+y}$. Those elements $a, e, x, y$ are fixed for the rest of the proof. Moreover, we set $\tau = \mu_e$. 

Given $b \in V_n$, we set $\overline{b}=b +U_{n+2}$. Moreover, we set $\FF=V_{n+1}/U_{n+2}$
Since $\FF$  and $V_n/U_{n+1}$ are isomorphic $H$-modules, by 
Lemma~\ref{quadratic}(ii), there is a multiplication $\cdot$ on $\FF$ such that $(\FF,+,\cdot)$ is a field
with neutral element $\overline{e}$, $-\overline{b}^{-1} \equiv \overline{b\tau} \mod  U_{n+1}/U_{n+2}$ 
and $\overline{bh_c} \equiv \overline{b} \cdot \overline{c}^2 \mod  U_{n+1} /U_{n+2}$ for
all $b,c \in V_{n+1}, b\not\in U_{n+1}$.

Since $H_n$ normalizes $V_n$ we have $b\tau = -b\mu_b \tau = -bh_b^{-1}  \in V_n\setminus U_{n+1}$ 
for $b \in V_n \setminus U_{n+1}$. Therefore by Lemma~\ref{tau} the map $\tau$ induces a permutation on 
$(V_n/U_{n+2}) \setminus (U_{n+1}/U_{n+2})$ which we will also call $\tau$. 

We will prove the following claim:
$$\bar b \tau = - \bar b^{-1} \hspace{1cm} \text{for all } b \in V_{n+1} \setminus U_{n+1}.$$
As observed above, this claim implies that $V_{n+1}$ is invariant under $H_{n+1}$, which completes the proof of the proposition.

Since $h$ and $h_a =\mu_e \mu_a$ induce the same map on 
$V_n/U_{n+1}$, the element $\overline{a}$ generates the multiplicative group 
of $\FF$. Since $U_i/U_{i+1}, V_n/U_{n+1}$ and $\FF$ are isomorphic $H$-modules, $H$ has at most
two orbits on $\FF \setminus \{\overline{0}\}$ (namely the squares and the non-squares) with 
representatives $\overline{e}$ and 
$\overline{a}$. Therefore, in order to compute $\bar b \tau$ for all $\bar b \in \FF^\#$, it suffices to 
compute $e h^j \tau$ and $a h^j \tau$ for all $j \geq 0$.

We have the following sequence of congruences modulo $U_{n+2}$:
\begin{align*}
eh^j\tau  &\equiv  e\mu_{e+x}^2 h^j \mu_{e+x}^2 \tau \\
&\equiv  (e+x-x)\tau^2 \mu_{e+x} h^{-j} \mu_{e+x} \tau  \hspace{2cm} \text{because $\mu_{e+x}$ inverts $h$}\\
&\equiv  ((e+x)\tau-xh_{e+x}^{-1})h_{e+x}  h^{-j} h_{e+x}^{-1}   \hspace{2cm} \text{by Lemma~\ref{tau}
}\\
& \equiv (-(e+x)-x)h^{-j} h_{e+x}^{-1} \hspace{2cm}\text{by Lemma~\ref{mumaps}(iv) } \\ 
&\equiv  (-e-2x)h^{-j} h_{-e+x} \hspace{1cm} \text{because $(e+x)\tau \equiv -e +x \mod U_{n+2}$ by 
Lemma~\ref{tau}}\\
& \equiv   -eh^{-j}h_{e-x} -2x h^{-j} \hspace{2cm} \text{by  Proposition~\ref{Hua maps II}
and Lemma ~\ref{mumaps}(i)}\\
&\equiv   -eh^{-j}+x h_{-eh^{-j},e\tau} h_e -2x h^{-j} \hspace{1cm} \text{by  Lemma~\ref{Hua maps} and because $e h^{-j}h_x \in U_{n+2}$}\\
&\equiv   -eh^{-j} -2x h^{-j}  +x h_{eh^{-j},e}.  
\end{align*}

Set $\EE=
U_{n+1}/U_{n+2}$.   
By  Lemma~\ref{quadratic}(iv), there is a multiplication $\cdot$ on $\EE$ such that 
$(\EE,+,\cdot)$ is a field and such that there is a natural number $0 \leq s \leq 
\frac{r_i}{2}$ and an embedding $\iota$ from $\FF$ to the algebraic 
closure of $\EE$ satisfying $\overline{ch_b} =\overline{c}\cdot \iota(\overline{b})^{1+r}$ 
for $r=p^s$.
Since $\bar e h^j  = \bar a^{2j}$, we deduce from the computation above that 
$$\overline{a}^{2j} \tau =-\overline{a}^{-2j} -2\ox 
\iota(\overline{a})^{-2j(1+r)}+\ox\iota(\oa^{-2jr}+\oa^{-2j} ).$$

We now compute a similar formula for the non-squares of $\FF$. 
Notice that $\mu_{a+y}h^j \mu_{e+x} =h^{-1-j}$. Therefore, by a similar argument, we obtain  the following 
sequence of congruences modulo $U_{n+2}$:
\begin{align*}
a h^j \tau 
&\equiv  (a+y-y)  \mu_{a +y}^2 h^j \mu_{e+x} ^2 \tau \\
&\equiv (a+y-y)\tau h_{a+y} h^{-j-1} \mu_{e+x} \tau \hspace{2cm} \text{since $\mu_{a+y}h^j \mu_{e+x} 
=h^{-1-j}$}\\
& \equiv ((a+y)\tau  -yh_{a+y}^{-1}) h_{a+y} h^{-j-1}h_{e+x}^{-1} \hspace{2cm} \text{by Lemma ~\ref{tau}} 
\\
& \equiv ( (a+y) \mu_{a+y} -y )h^{-j-1} h_{e+x}^{-1} \\
&\equiv  (-a-y -y ) h^{-j-1} h_{e+x}^{-1} \hspace{2cm} \text{by Lemma~\ref{mumaps}(iv)} \\
&\equiv   -a h^{-j-1} h_{e+x}^{-1} -2yh^{-j-1} \hspace{2cm} \text{by Proposition 
~\ref{Hua maps II}}\\ 
&\equiv  -a h^{-j-1} h_{-e+x} -2y h^{-j-1} \hspace{2cm}
 \text{because $(e+x)\tau \equiv -e +x \mod U_{n+2}$ by ~\ref{tau}}  \\
 & \equiv -ah^{-j-1} h_{e-x} -2yh^{-j-1} \hspace{2cm} \text{by Lemma \ref{mumaps}(i))}\\
&\equiv   -a h^{-j-1} +x h_{-a h^{-1-j},e\tau}h_e -2y h^{-j-1} \hspace{0.5cm} \text{by  
Lemma ~\ref{Hua maps} and because $e h^{-j}h_x \in U_{n+2}$} \\
&\equiv  -ah^{-j-1} +x h_{-ah^{-j-1},-e}-2y h^{-j-1} 
\end{align*}
Since $\bar a h^j  = \bar a^{2j+1}$, this implies  
$$\overline{a}^{2j+1}\tau =-\overline{a}^{-2j-1} -2\oz \iota(a)^{-(2j+1)\frac{1+r}{2} }+ 
\ox\iota(\overline{a}^{-(2j+1)r} +\overline{a}^{-(2j+1)})$$
where $\oz = \oy\iota(\overline{a})^{-\frac{1+r}{2}}$.

Combining the two conclusions of those two computations, we deduce that for any  $\ob \in \FF \setminus \{\overline{0}\}$, we have
\renewcommand{\theequation}{$*$}
\begin{equation}\label{eq:*}
\overline{b}\tau = -\overline{b}^{-1} -2
\oz_{\ob} \iota(\ob)^{-\frac{1+r}{2}} +\overline{x}\iota(\ob^{-r} +\ob^{-1}),
\end{equation} 
where  $\oz_{\ob}=\ox$ if $\ob$ is a square in $\FF$, 
and $\oz_{\ob}=\oz$ if $\ob$ is a non-square in $\FF$.

\medskip
Since $(b+c)\tau \equiv b\tau +ch_b^{-1} \mod  U_{n+2}$ for $b \in V_{n+1}^*$ and 
$c \in U_{n+1}$ by Lemma~\ref{tau}, we have
$$
(\ob +\overline{c})\tau =-\overline{b}^{-1}-2z_b
\iota(\overline{b})^{-\frac{1+r}{2}}+\overline{x}
\iota(\overline{b}^{-1}+\overline{b}^{-r})+\overline{c}\iota(\overline{b})^{-1-r}.
$$
By \cite[Proposition~4.1.1]{DS1}, we have $\tau =\mu_e =\alpha_e \alpha_{-e\tau^{-1}}^{\tau} \alpha_{-(-e\tau^{-1})\tau}
= \alpha_e \alpha_e^{\tau} \alpha_e$. In particular,  for $b \in V_{n+1}^*$, we have $\bar b \tau =   ((\ob +\ove)\tau +\ove)\tau +\ove$. Therefore, for any   $b \in V_{n+1}$ with $\ob \ne -\ove, \overline{0}$, we obtain:
\begin{align*}
\lefteqn{-\ob^{-1}  -  2z_{\ob} \iota(\ob)^{-\frac{1+r}{2}} +\ox (\iob^{-r}+\iob^{-1})} \qquad \qquad \\
&=   \ob \tau \\
& =  ((\ob +\ove)\tau +\ove)\tau +\ove\\
& =  (-(\ob +\ove)^{-1}-2\oz_{\ob+\ove}\iota(\ob+\ove)^{-\frac{1+r}{2}}+\ox
(\iota(\ob+\ove)^{-r}+\iota(\ob+\ove)^{-1}) +\ove)\tau +\ove \\
& = 
(\ob(\ob +\ove)^{-1}  -2\oz_{\ob+\ove}\iota(\ob+\ove)^{-\frac{1+r}{2}}+\ox
(\iota(\ob+\ove)^{-r}+\iota(\ob+\ove)^{-1}))\tau +\ove \\
&=  -\ob^{-1}(\ob +\ove) -2\oz_{\ob+\ove}\iota(\ob+\ove)^{-\frac{1+r}{2}} \iota(\ob)^{-1-r} 
\iota(\ob +\ove)^{1+r} \\
& \qquad + \ox(\iota((\ob +\ove)^{-r}  +(\ob +\ove)^{-1} )) \iota(\ob)^{-1-r}\iota(\ob +\ove)^{1+r} \\
& \qquad -2\oz_{\ob(\ob +\ove)}\iota(\ob)^{-\frac{1+r}{2}} \iota(\ob +\ove)^{\frac{1+r}{2}} 
 +\ox \iota(\ob^{-r}((\ob +\ove)^r + \ob^{r-1}(\ob +\ove )))+\ove \\
&=  -\ob^{-1} -2\oz_{\ob +\ove} \iota(\ob +\ove)^{\frac{1+r}{2}} \iota(\ob)^{-1-r} + \ox \iota(2\ob^{1+r} +2\ob +2\ove +2\ob^r) \iota(\ob)^{-1-r} \\
& \qquad - 2\oz_{\ob(\ob +\ove)} \iota(\ob)^{-\frac{1+r}{2}} \iota(\ob +\ove)^{\frac{1+r}{2}}. 
\end{align*}
It follows that 
\begin{align*}
f(\ob) &=\ox\iota(2\ob^{1+r} +\ob^r + \ob + 2\ove) + 2\oz_{\ob}\iota(\ob)^{\frac{1+r}{2}} -
2\oz_{\ob +\ove}\iota(\ob +\ove)^{\frac{1+r}{2}}-2\oz_{\ob(\ob +\ove)}\iota(\ob)^{\frac{1+r}{2}} 
\iota(\ob +\ove)^{\frac{1+r}{2}}\\
& =\overline{0}
\end{align*}
for all $\ob \in \FF \setminus \{-\ove,\overline{0} \}$. 

\medskip
Assume that $\chara(\FF) =2$. If $\bar b = \ove = -\ove$, then $\ob\tau =
 -\ob^{-1}$ by (\ref{eq:*}). Otherwise we have $\ob \in \FF \setminus \{-\ove,\overline{0} \}$, and  we get 
$\ox (\iob +\iob^r) = \overline{0}$ from the equation above, so that  $r=1$ or $\ox=\overline{0}$. 
In both cases we have $\ob\tau =
 -\ob^{-1}$ by (\ref{eq:*}), and the desired claim follows. 
 
\medskip 
We assume  henceforth that  $\chara(\FF) \ne 2$. We shall prove that $\ox = \oz$, and that the common value is $\bar 0$ unless $r=1$. In either case, the desired equation $\bar b \tau = - \bar b^{-1}$ follows readily from (\ref{eq:*}), which finishes the proof. 


As a first case, assume that $r=1$. If $\FF = \FF_3$, we apply the equation $f(\ob) = \bar 0$ with $\ob = \ove$; it yields $\ox = \oz$, as claimed. Otherwise, there exists a non-square $\ob$ in $\FF \setminus \{-\ove,\overline{0} \}$. In that case, the equation $f(\ob) = \bar 0$ yields
 $$-2\oz_{\ove+\ob} \iota(\ove+\ob) +\ox\iota(2\ove+2\ob+2\ob^2)-2\oz_{\ob(\ove+\ob)} \iota(\ob^2+\ob)+
  2\oz \iob=\overline{0}.$$
Therefore, if $\ob+\ove$ is a square, we get 
 $ 2\ox \iob^2-2\oz \iob^2=0$, and thus $\ox=\oz$, as claimed. Similarly, if 
 $\ob+\ove$ is a non-square, we get $2\ox \iota(\ove +\ob) - 2 \oz \iota(\ove + \ob)= \bar 0$, and again $\ox=\oz$, as claimed. 
 
We now consider the remaining case $r \geq 3$. The goal is then to show that $\ox=\oz= \bar 0$.

Suppose for a contradiction that $\ox \neq  \bar 0$. Let  $1=\iota(\ove)$ denote the unit element of $\EE$, and $Y$ be an indeterminate. We define four polynomials $f_1(Y), \dots, f_4(Y) \in \EE[Y]$ as follows. The polynomial  $f_1(Y)$ is obtained by replacing $\iob$ by $Y$, and   $\oz_{\ob},\oz_{\ove+\ob}$ and $\oz_{\ob(\ove +\ob)}$ by $\ox$. The polynomial  $f_2(Y)$ is obtained by replacing    $\oz_{\ob}$ by $\ox$, and $\oz_{\ove+\ob}$ and $\oz_{\ob(\ove+\ob)}$ by $\oz$. The polynomial  $f_3(Y)$ is obtained by replacing    $\oz_{\ob}$ and $\oz_{\ob(\ove+\ob)}$ by $\oz$, and $\oz_{\ove+\ob}$ by $\ox$. The polynomial  $f_4(Y)$ is obtained by replacing $\oz_{\ob}$ and 
 $\oz_{\ove+\ob}$ by $\oz$ and $\oz_{\ob(\ove+\ob)}$ by  $\ox$. One computes that:
\begin{align*}
f_1(Y) &= \ox \bigg(  2 Y^{r+1} + Y^r + Y + 2 - 2(Y^2 + Y)^{\frac{1+r} 2} - 2 (Y + 1)^{\frac{1+r} 2} +2 Y^{\frac{1+r} 2}  \bigg),\\
f_2(Y) & =  \ox \bigg(   2 Y^{r+1} + Y^r + 2 Y^{\frac{1+r} 2}+  Y + 2\bigg) -2 \oz   (Y + 1)^{\frac{1+r} 2} (Y^{\frac{1+r} 2}  +1), \\
f_3(Y) & =  \ox \bigg(   2 Y^{r+1} + Y^r - 2 (Y+1)^{\frac{1+r} 2}+  Y + 2\bigg) -2 \oz  Y^{\frac{1+r} 2}  \bigg( (Y + 1)^{\frac{1+r} 2} -1\bigg), \\
f_4(Y) &= \ox \bigg(   2 Y^{r+1} + Y^r - 2 (Y^2+Y)^{\frac{1+r} 2}+  Y + 2\bigg) -2 \oz   \bigg( (Y + 1)^{\frac{1+r} 2} - Y^{\frac{1+r} 2}\bigg).
\end{align*}
Observe that all four polynomials are non-zero. Indeed, the term of highest degree in $f_2(Y)$ and $f_3(Y)$ is $2(\ox - \oz) Y^{r+1}$. Moreover, we have $f_4(\bar 0) = 2(\ox - \oz)$. For $f_1$, we see that the term of highest degree is $\ox \frac{(r+1)(r-1)} 4 Y^{r-1}$ if $r-1 > \frac{r+1} 2$, i.e. if $r > 3$. On the other hand, if $r=3$, then $\chara(\FF) = p = 3$ since $r$ is a power of $p$ by definition. In that case we have $f_1(Y) = \ox Y^2$. This confirms that all four polynomials are non-zero.

Since $f(\bar b) = \bar 0$ for all $\ob \in \FF \setminus \{-\ove,\overline{0} \}$ by the above, we deduce that the polynomial
$$P(Y) = f_1(Y) f_2(Y) f_3(Y) f_4(Y) (Y^2 + Y) \in \EE[Y]$$
vanishes on the field $\iota(\FF)$. Since $f_1$ and $f_4$ have degree~$\leq r-1$ while $f_2$ and $f_3$ have degree~$r+1$, we see that $P$ has degree~$\leq 4r+2$. Therefore $|\FF| = q_i \leq 4r+2 < 5r \leq 5 \sqrt{q_i}$. This implies that $q_i = 9$ and $r = 3$, so that $f_1(Y) = \ox Y^2$. For
 $\bar b = \bar e$ both $\bar b$ and $\ove + \bar b$ are squares. We must then have $f_1(\ob) = \bar 0$, which yields $\ox \iota(\ob)^2 = \bar 0$, a contradiction. This proves that $\ox = \bar 0$. 

Assume now that $\oz \neq \bar 0$. We now consider the polynomial $Q(Y) = f_3(Y) f_4(Y)$, so that $Q(\iota(\bar b)) = \bar 0$ for all non-squares $\ob \in \FF \setminus \{-\ove,\overline{0} \}$. Now $Q$ has degree $r+1+ {\frac{r+1}{2}}-1$, hence 
$\frac{3}{2} r +\frac{1}{2} \geq \frac{1}{2} (q_i-1)$. We deduce that
$q_i \leq 3r +2 < 4r \leq 4\sqrt{q_i}$, and as before it follows that $q_i =9$ and $r=3$. In that case, one computes $Q(Y) = - \oz Y^3 (Y-1)^2$. This implies that every non-square in $\FF$ must be zero, which is absurd. Hence $\oz = \bar 0$, and the proof is complete.
\end{proof}

\begin{lemma}\label{faithful action}
Assume that the root group $U_\infty$ is a closed subgroup of $\Aut(T)$. Let $V, W \leq U$ subgroups such that $U_0 =U_1 +V$ and $U_1 = U_2 +W$. Then the centraliser $C_{\HH}(V+W) $ is trivial.
\end{lemma}
\begin{proof}
Let $h \in C_{\HH}(V+W)$. Then $C_U(h)$ is a root subgroup of $U$, see Lemma 6.2.3 from \cite{DS1}. Let 
$a \in V \setminus U_1$ and $b \in  W \setminus U_2$ and set $t =\mu_a \mu_b$. Then $t$
conjugates $U_n$ to $U_{n+2}$ for all $n \in \ZZ$. Since $a,b \in C_U(h)$ and since
$C_U(h)$ is a root subgroup, it follows that $t$ normalizes $C_U(h)$. Therefore the set
$$\big\{\sum_{i=n}^m a_it^i +b_i t^i \; | \; a_i \in V, b_i \in W; n,m \in \ZZ, n<m\big\}$$ 
is contained in $C_U(h)$. Identifying $U$ with $U_\infty$ via $\alpha$, we may view $U$ as a topological group and infer that $C_U(h)$ is dense in $U$. Since $h$ induces a continuous automorphism on $U$, it follows that $C_U(h)$ is closed, so that $C_U(h) =U$ and $h=1$.
\end{proof}

\begin{prop}\label{prop:FiniteRootSubgroup}
Assume that $T$ is locally finite and that $U$ is abelian of exponent $p$ for some prime $p$. Assume further that the root group $U_\infty$ is closed in $\Aut(T)$. 

For any  finite solvable $p^{\prime}$-subgroup $X \leq \HH$ and  any $i \in \{0,1\}$, there is a finite subgroup $V \leq U_i$ satisfying the following conditions:

\begin{enumerate}
\item $V$ is a root subgroup. 

\item $U_i $ is the direct sum of $V$ and $U_{i+1}$. 

\item The group $\tilde{\mathcal H} = \la \mu_a \mu_b \; | \; a, b \in V^\#\ra$ is finite cyclic of order prime to $p$.

\item $V$ is normalised by $X$ and $\tilde{\mathcal H}$.
\end{enumerate}
\end{prop}

\begin{proof}
Let $U_i=V_0 \geq V_1 \geq V_2 \geq \ldots$ be a descending chain afforded by applying  Proposition~\ref{construction of root groups}. Remark that $U_n$ is an open subgroup of $U$ for all $n$. Since $V_n \geq U_{n+1}$ for all $n$, we infer that $V_n$ is also open, hence closed, for all $n$. 

Set $V =\bigcap_{n \geq 0} V_n$. 
For every $a \in U_i \setminus U_{i+1}$ and   $n \geq 0$, we have 
$U_{i+1} a \cap V_n \ne \varnothing$. By compactness, it follows that $V \cap U_{i+1}a = \bigcap_{n \in \NN} (V_n \cap U_{i+1} a)$ is non-empty. Therefore $U_i = V + U_{i+1}$. Moreover, we have   
$V_n \cap U_{i+1} =U_{n+1}$ for all $n$, so that $V \cap U_{i+1} = \bigcap_{n \in \NN} U_n =0$. Thus 
$V$  is a complement for $U_{i+1}$ in $U_i$, thereby proving (ii). Moreover, since we have  $\tilde{\mathcal H} \leq N_{\HH}(V_n)$ 
for all $n \geq 0$, we see that $\tilde{\mathcal H} \leq N_{\HH}(V)$. Since $X$ normalises $V_n$ for all $n$, we also obtain (iv). 

Let $e_n \in V_n$ be the element afforded by Proposition~\ref{construction of root groups}(iii); let $e \in V$ be any accumulation point of the sequence $(e_n)$.  Then the element $\tau = \mu_e$ is an accumulation point of the sequence $\mu_{e_n}$ by \cite[Proposition~3.8]{CDM}. The fact that $a \tau \in V$ for all $a \in V^\#$ now follows from Proposition~\ref{construction of root groups}(iii), in view of the definition of $V$.

Remark that the Moufang subset $(V, \mu_a)$ with $a \in V^\#$ is isomorphic to the local Moufang set at $x_i$. It is thus isomorphic to $\mathbb  M(\FF_{q_i})$ by Lemma~\ref{local situation}, we deduce from  Lemma~\ref{root subgroup} that $\tilde{\mathcal H}$ is cyclic of order $q_i-1$ if $p=2$ (resp.   
$\frac{q_i-1}{2}$ if $p$ is odd).
\end{proof}

\subsection{Construction of a twin tree lattice}

RGD-systems and their relations with twin tree lattices have been reviewed in Section~\ref{sec:RGD}. The main step in our proof of Theorem~\ref{thm:main} is given by the following.

\begin{theorem}\label{RGD}
Let $T$ be a locally finite tree and $p$ be a prime. Let  $(U_{\mathfrak{e}} \; | \; {\mathfrak{e}} \in \partial T)$ be a collection of closed, abelian subgroups of exponent $p$  in $\Aut(T)$ such that $(\partial T, (U_{\mathfrak{e}} \; | \; {\mathfrak{e}} \in \partial T))$ is a Moufang set. Let $G = \overline{\la U_{\mathfrak{e}} \; | \; {\mathfrak{e}} \in \partial T\ra}$ be the closure of its little projective group. 

Then, for any pair of distinct ends $+{\mathfrak{e}}, -{\mathfrak{e}} \in \partial T$,  there is a dense subgroup $\Gamma \leq  G$, a finite subgroup $X$ of the Hua group, and a collection of finite root subgroups $\{U_n^{\epsilon} \leq U_{\epsilon {\mathfrak{e}}}\}_{n \in \ZZ, \epsilon \in\{+,-\}}$ such that $(\Gamma, (U_n^{\epsilon})_{n \in \ZZ, \epsilon \in\{+,-\} },X)$ is an RGD-system.

\end{theorem}
\begin{proof}
Let $V^0 \leq U_0$ be the root subgroup afforded by  applying  Proposition~\ref{prop:FiniteRootSubgroup} with $i=0$ and $X=1$. Set $Y_0 =   \la \mu_a \mu_b \; | \; a, b \in V^\#\ra$. Then $Y_0$ is a finite solvable $p'$-subgroup of $\HH$. We then  apply Proposition~\ref{prop:FiniteRootSubgroup}   with $X=Y_0$ and $i=1$. We obtain a $Y_0$-invariant root subgroup $W^0$ of $U_1$ which is a complement for $U_2$ in $U_1$. Set 
$Z_0 = \langle \mu_a \mu_b \; | \; a,b \in (W^0)^\#\ra$. Since $Y_0$ normalizes $W^0$, it also normalizes 
$Z_0$. Since $U_1/U_2$ and $W^0$ are isomorphic $Y_0$-modules and since $\HH/C_{\HH}(U_1/U_2)$ is a $p^{\prime}$-group, we see that $X_0 =Y_0 Z_0$ 
is a finite solvable $p^{\prime}$-group. 

We now apply Proposition~\ref{prop:FiniteRootSubgroup} once more, with $X=X_0$ and $i=1$. This yields  a $X_0$-invariant root subgroup $V^{1}$ which is a complement for $U_1$ in  $U_0$. Set $Y_1=\la \mu_a \mu_b \; | \; a,b \in (V^{1})^{\#}\ra$. As above we conclude that 
$Y_1$ is normalized by $X_0$. Set $X_1 =Y_1 X_0$. We repeat this process and get an 
ascending series of finite solvable $p^{\prime}$-groups $X_0 \leq X_1 \leq X_2 \leq X_3 \leq \ldots$ such that $X_i$ 
normalizes a complement $V^i$ of $U_1$ in $U_0$ and a complement $W^i$ of $U_2$ in $U_1$, such that $V^i$ and $W^i$ are both root subgroups. Since 
$X_i$ acts faithfully on $V^i +W^i$ by Lemma~\ref{faithful action}, this chain must stabilize and the process terminates. We therefore obtain  a finite $p^{\prime}$-group $X$ which normalizes a complement $V$ for $U_1$ in $U_0$ and a complement 
$W$ for $U_2$ in $U_1$. Moreover  $V$ and $W$ are root subgroups,  and the sets 
$\{\mu_a \mu_b \; | \;  a,b \in V^{*}\}$ and $\{\mu_a \mu_b \; | \; a,b \in W^{*}\}$ are contained in $X$. Notice that 
if $a \in V^{*} \cup W^{*}$ and $x \in X$, then $x^{\mu_a} = \mu_a x \mu_a = x \mu_{ax} \mu_a
\in X$.

Now we fix $a \in V^{*}$, $b\in W^{*}$ and set $s_0 =\mu_a$, $s_1=\mu_b$ and $t =s_0 s_1$. We define 
$U_0^+ = \alpha(V), U_1^+ =\alpha(W), U_0^- =(U_0^+)^{s_0}, U_{-1}^-=(U_1^+)^{s_1}, U_{2n+i}^+ = (U_i^+)^{t^n}$ 
and $U_{2n+i}^- = (U_i^-)^{t^n}$ for $n \in \ZZ$ and $i \in \{0,1\}$. Note that 
$t$ normalizes $X$ and thus $X$ normalizes $U_n^{\epsilon}$ for $n \in \ZZ$ and 
$\epsilon \in \{ +,- \}$.
Let $\Gamma$ be the subgroup of $G$ generated by $X$ and $ U_n^{\epsilon}$ for all $n \in \ZZ$ and $\epsilon \in\{+,-\}$. We proceed to verify that   $(\Gamma, (U_n^{\epsilon})_{n \in \ZZ, \epsilon \in\{+,-\} },X)$ is 
an RGD-system by checking the axioms successively. 

\medskip \noindent (RGD0) 
$V$ and $W$ are non-trivial and so is $U_n^{\epsilon}$ for all $n \in \ZZ$ and  $\epsilon \in\{+,-\}$.

\medskip \noindent (RGD1) 
Since $U$ is abelian, $[U_n^{\epsilon}, U_m^{\epsilon}] =1$ for $n,m \in \ZZ$ and 
$\epsilon \in\{+,-\}$. Thus (RGD1) follows.

\medskip \noindent (RGD2) 
If $c\in U_0^+$, then $s_0 \mu_c \in X$, so $\mu_c$ induces a fundamental 
reflection of the root system. The same holds for $U_1^-$, so (RGD2) follows.

\medskip \noindent (RGD3)  
 Let $U_+ = \langle U_n^+ \; | \; n \in \ZZ \rangle$. 
Since $U \cap U^{s_i} =1$ for $i=0,1$, we have $U_n^{-} \cap U_+ \leq 
U_n^- \cap U =1$ for all $n \in \ZZ$, so (RGD3) holds.

\medskip \noindent (RGD4) 
By definition the group $\Gamma$ is generated by $X$ and 
$(U_n^{\epsilon})_{n\in\ZZ, \epsilon \in\{+,-\} }$.

\medskip 
This confirms that $(\Gamma, (U_n^{\epsilon})_{n \in \ZZ, \epsilon \in\{+,-\} },X)$ is an RGD-system.
\end{proof}

\subsection{Endgame}

We now assemble all the results collected thus far to complete the proof of our main theorem. 

\begin{proof}[Proof of Theorem~\ref{thm:main}]
Let $\infty, \mathfrak{o}$ be two different elements of $\partial T$.
By \cite[Proposition~7.2.2]{DS1}, the root group $U_\infty$ is either torsion-free and uniquely divisible, or  of exponent $p$ for some prime $p$. In the former case, the desired conclusion follows   from \cite[Theorem~B]{CDM}. We assume henceforth that $U_\infty$ is of exponent $p$. 

By Theorem~\ref{RGD} there are subgroups
$U_n^{+} \leq U_{\infty}$ and $U_n^{-} \leq U_{\mathfrak o}$, a dense subgroup $\Gamma \leq G = \overline{\la U_{\infty} \cup U_{\mathfrak o} \ra}$, and a subgroup 
$H \leq \mathfrak{H}$ such that $(\Gamma, (U_n^{\epsilon})_{n \in \ZZ, \epsilon \in \{+,-\}}, H)$ is a 
RGD-system. Since the root groups are abelian and in view of Lemma~\ref{local situation}, we see that the corresponding Moufang twin tree is a commutative
$\SL_2$-twin tree. The desired conclusion now follows from Theorem~\ref{ClassificationTheorem} and Remarks~\ref{rem:Classif}.
\end{proof}

\section{Applications}

\subsection{Boundary-transitive tree automorphism groups}

\begin{proof}[Proof of Corollary~\ref{cor:con}]
Suppose that (i) holds.
Let $\infty \in \partial T$ be the repelling fixed point of $h$. 
By Corollary~2.6 from \cite{CDM}, the group $\con(h)$ is closed, so by Theorem~C of loc.~cit. $T$ is boundary Moufang with
the conjugates of $\con(h)$ as root groups. Let $G^{\dagger}$ be the normal subgroup of $G$ generated 
by the conjugates of $\con(h)$. 
By Theorem~\ref{thm:main} there is a non-Archimedean local field
$k$, a simply connected, $k$-simple algebraic group $\mathbf{G}$ defined over $k$ and an 
epimorphism $\varphi: \mathbf{G}(k) \to G^{\dagger}$ with $ker(\varphi) = Z(\mathbf{G}(k))$. Moreover  $G^\dagger$ is closed in $\Aut(T)$. 
By Proposition~3.5 from \cite{CDM}, it follows that   $G^{\dagger} = G^{(\infty)}$. This proves that  (ii) holds. 

The other direction follows from the  well-known properties of rank-one simple algebraic groups over local fields. 
\end{proof}

\begin{proof}[Proof of Corollary~\ref{cor:unimod}]
By Theorem~8.1 from \cite{CCMT}, the group $G$ is either an almost connected rank one simple Lie group, or $G$ has a continuous proper, faithful action by automorphisms on a thick locally finite tree $T$, which is doubly transitive on $\partial T$. In the former case, the desired conclusion holds since all simple Lie groups are algebraic. In the latter case, one observes that the given element $h$ must act as a hyperbolic automorphism of $T$, since otherwise its contraction group $\con(h)$ would be trivial, forcing $G$ to be compact, hence trivial. Therefore, the desired conclusion follows from Corollary~\ref{cor:con}.
\end{proof}

\subsection{Sharply $3$-transitive groups on compact sets}

\begin{proof}[Proof of Corollary~\ref{cor:sharply3transitive}]
Assume that $G$ is a $\sigma$-compact group acting 
sharply $3$-transitively and continuously on a compact set $\Omega$. By Theorem B of \cite{CarDre}, the group $G$ is a rank one simple Lie group, in which case the conclusion follows from \cite{Tits3trans} (and the condition (i) is then automatically satisfied), or there is 
a locally finite tree $T$ on which $G$ acts properly, vertex-transitively and continuously by 
automorphism and a $G$-equivariant homeomorphism $f: \Omega\to \partial T$. We stick henceforth to the latter case. 

Assume that (i) holds. Let $\omega \in \Omega$ and $U_\omega \leq G_\omega$ be a normal subgroup acting regularly on $\Omega \setminus \{\omega\}$. It follows that any two non-trivial elements of $U_\omega$ are conjugate in $G_\omega$. Given an involution $\tau \in G_\omega$ and any element $u \in U_\omega$, we see that the commutator $[\tau, u]$ is a product of two involutions which belongs to $U_\omega$. Therefore every element of $U_\omega$ is a product of two involutions. This implies that the the group $U_\omega$ must be abelian (see \cite[Theorem~3.7]{Kerby} for more details). Therefore, so is its closure $\overline{U_\omega}$. Since any transitive action of an abelian group is regular, it follows that $\overline{U_\omega}$ also acts regularly on $\Omega \setminus \{\omega\}$. This implies that $U_\omega = \overline{U_\omega}$. In other words $U_\omega$ is a closed subgroup of $G$. 

The hypotheses of Theorem~\ref{thm:main} are thus satisfied, and we infer that  the little projective group $G^{\dagger}=\la U_{\omega} \; | \; \omega \in \Omega \ra$ is of the form $\mathbf G(k)/Z$, where $\mathbf G$ is a rank one algebraic group with 
abelian root groups over a local field $k$, and  that $T$ is isomorphic to the Bruhat--Tits tree of $\mathbf G(k)$. As recalled in the introduction, we must have either $\mathbf G(k) \cong \SL_2(D)$, where $D$ is a finite-dimensional central division algebra over $k$, or   $\mathbf G(k) \cong \SU_2(D, h)$, where $D$ is the quaternion central division algebra over $k$, and $h$ is an antihermitian sesquilinear form of   Witt index~$1$ relative to an involution $\sigma$ of the first kind such that the space of symmetric elements $D^\sigma$ has dimension~$3$. Direct computations show that the only group $\mathbf G(k)$ on that list satisfying the condition that the pointwise stabiliser of any triple of distinct points of $\partial T$  acts trivially, is  the group $\mathbf G(k) \cong \SL_2(k)$. This implies that  $G$   must 
be isomorphic to a subgroup of $\Aut(\PSL_2(k)) = \PGaL_2(k)$, thereby proving (ii). 

That (ii) implies (i) is again clear, the regular normal subgroups being afforded by the maximal unipotent subgroups of $\PSL_2(k)$.
\end{proof}

We note here that $\PGaL_2(k)$ may contain other sharply $3$-transitive subgroups than $\PGL_2(k)$. For example, 
if $\sigma$ is an involutory Galois automorphism of $k$ and $\varphi: k^* \to \la \sigma \ra$ is a homomorphism with
$\varphi(a^{\sigma}) =\varphi(a)$ for all $a \in k^*$, then the group 
$$G =\big\{\varphi(\det g) g \; | \; g \in GL_2(k)\big\}/
Z(GL_2(k))$$ 
acts sharply $3$-transitive on $\mathbb{P}^1(k)$.

\medskip
In general, one can define the \textbf{permutational characteristic} of a sharply $2$-transitive group $G$ on a set $\Omega$ as follows (see also \cite[\S 3]{Kerby}). Let $J$ be the collection of all involutions in $G$, which is clearly a non-empty single conjugacy class in $G$. If the elements of $J$ have no fixed point in $\Omega$, then one sets $\chara G  =2$. If that is not the case, then one  proves that the set $J^{2*}$ consisting of all elements of $G$ that are the product of two distinct involutions, forms a conjugacy class $G$, the order of whose elements is  either infinite, or an odd prime $p$. One sets $\chara G= 0$ or $\chara G = p$ accordingly.  If $G$ is sharply $3$-transitive on   $\Omega$, then  the stabiliser $G_\omega$ of a point $\omega \in \Omega$ is sharply $2$-transitive on $\Omega \setminus \{\omega\}$, and  the assignment $\chara G  =\chara G_{\omega}$ defines the permutational characteristic of $G$.

The results in \cite[\S13]{Kerby} show that if $G$ is sharply $3$-transitive of characteristic $p= 3$ or $
p \equiv 1 \mod 3$, then the point stabiliser $G_\omega$ contains a normal subgroup $U_\omega$ acting regularly on $\Omega \setminus \{\omega\}$. Therefore Corollary~\ref{cor:sharply3transitive} yields the following.

\begin{coro}\label{cor:1modulo3} 
Let $G$ be a $\sigma$-compact locally compact group and $\Omega$ be a compact $G$-space, such that the $G$-action on $\Omega$ is sharply $3$-transitive. If $\chara G = 3$ or $ \chara G  \equiv 1 \mod  3$, then there is a non-Archimedean local field $k$ and 
a group $H$ with $\PSL_2(k) \leq H \leq \PGaL_2(k)$ such that $(G,\Omega)$ and $(H,\mathbb{P}^1(k))$ are isomorphic 
permutation groups.   
\end{coro}

\end{document}